\documentclass{article}
\usepackage[a4paper, left=4.0cm, right=4.0cm, top=4cm]{geometry}
\usepackage{amsmath, amssymb}
\usepackage{hyperref,cleveref}
\usepackage{amsthm}
\usepackage{bm}
\usepackage{xcolor}
\usepackage{graphicx}
\usepackage{pgfplots}
\usepgfplotslibrary{groupplots}
\pgfplotsset{/pgfplots/group/.cd,
    horizontal sep= 2cm,
    vertical sep= 2cm
}
\pgfplotsset{
  every axis plot/.append style={line width=2pt,},
  every axis/.append style={
    ymajorgrids,
    xmajorgrids,
    grid style={dashed, lightgray,semithick},
   	axis line style = semithick,
   	every tick/.style={semithick,},
	  yticklabels={,,},
  },
}
\newcommand{\N}{{\mathbb{N}}}
\newcommand{\R}{{\mathbb{R}}}
\newcommand{\Cn}{{\mathbb{C}}}

\newcommand{\tr}{\textsf{tr}}
\newcommand{\rT}[1]{#1^{\textsf{T}}}
\newcommand{\rH}[1]{#1^{\textsf{H}}}
\newcommand{\rmH}[1]{#1^{-\textsf{H}}}

\newcommand{\si}[1]{{#1^{+}}}

\newcommand{\grad}[1][]{{\nabla_{#1}}}
\newcommand{\dd}[2][]{{\frac{\mathrm{d}^{#1}}{\mathrm{d}#2}}}

\newcommand{\ddt}{{\dd{t}}}

\newcommand{\Jtwo}[1]{{\mathbb{J}_{2#1}}}
\newcommand{\TJtwo}[1]{{\rT{\mathbb{J}}_{2#1}}}

\newcommand{\Jtk}{{\Jtwo{k}}}

\newcommand{\JtN}{{\Jtwo{N}}}
\newcommand{\TJtN}{{\TJtwo{N}}}

\newcommand{\I}[1]{{\fI_{#1}}}

\newcommand{\Z}[1]{{\fzero}_{#1}}

\newcommand{\tInit}{{t_{\mathrm{0}}}}
\newcommand{\tEnd}{{t_{\mathrm{end}}}}
\newcommand{\It}{I_t}
 
\newcommand{\fxInit}{{\fx_{\mathrm{0}}}}
\newcommand{\paramDomain}{\mathcal{P}}

\newcommand{\reduce}[1]{#1_{\mathrm{r}}}
\newcommand{\fxr}{\reduce{\fx}}

\newcommand{\solMnf}{\mathcal{M}}
\newcommand{\ns}{{n_{\mathrm{s}}}}
\newcommand{\xs}{\fx^{\mathrm{s}}}

\newcommand{\Xs}{{\fX_{\mathrm{s}}}}
\newcommand{\fYs}{\fY_{\mathrm{s}}}

\newcommand{\Ham}{\mathcal{H}}
\newcommand{\Hamr}{\reduce{\Ham}}

\newcommand{\fe}{{\bm{e}}}

\newcommand{\fp}{{\bm{p}}}
\newcommand{\fq}{{\bm{q}}}

\newcommand{\fx}{{\bm{x}}}

\newcommand{\fA}{{\bm{A}}}
\newcommand{\fB}{{\bm{B}}}
\newcommand{\fC}{{\bm{C}}}
\newcommand{\fD}{{\bm{D}}}
\newcommand{\fE}{{\bm{E}}}
\newcommand{\fF}{{\bm{F}}}

\newcommand{\fH}{{\bm{H}}}
\newcommand{\fI}{{\bm{I}}}

\newcommand{\fM}{{\bm{M}}}

\newcommand{\fP}{{\bm{P}}}
\newcommand{\fQ}{{\bm{Q}}}
\newcommand{\fR}{{\bm{R}}}

\newcommand{\fU}{{\bm{U}}}
\newcommand{\fV}{{\bm{V}}}
\newcommand{\fW}{{\bm{W}}}
\newcommand{\fX}{{\bm{X}}}
\newcommand{\fY}{{\bm{Y}}}
\newcommand{\fZ}{{\bm{Z}}}

\newcommand{\fmu}{{\bm{\mu}}}

\newcommand{\fxi}{{\bm{\xi}}}

\newcommand{\fomega}{{\bm{\omega}}}

\newcommand{\fSigma}{{\bm{\Sigma}}}

\newcommand{\fOmega}{{\bm{\varOmega}}}
\newcommand{\fzero}{\ensuremath{\bm{0}}}

\newcommand{\calC}{\mathcal{C}}

\newcommand{\VE}{\fV_\text{E}}
\newcommand{\VQ}{\fV_\text{Q}}
\newcommand{\VP}{\fV_\text{P}}
\newcommand{\Xc}{{\fX_\text{c}}}
\newcommand{\Vc}{\fV_\text{c}}
\newcommand{\Uc}{\fU_\text{c}}
\newcommand{\VY}{\fV_\fY}
\newcommand{\fSigmaY}{\fSigma_\fY}
\newcommand{\UY}{\fU_\fY}
\newcommand{\UB}{\fU_\fB}
\newcommand{\VB}{\fV_\fB}
\newcommand{\fSigmaB}{\fSigma_\fB}
\newcommand{\povs}{p_\mathrm{ovs}}
\newcommand{\qpow}{q_\mathrm{pow}}
\newcommand{\Ucr}{\fU_\mathrm{c}^\mathrm{r}}
\newcommand{\VrcSVD}{\fV_\mathrm{rcSVD}}
\newcommand{\VcSVD}{\fV_\mathrm{cSVD}}
\newcommand{\fSigmac}{\fSigma_\mathrm{c}}
\newcommand{\im}{\mathrm{i}}

\newcommand{\Qs}{\fQ_s}
\newcommand{\Ps}{\fP_s}
\newtheorem{theorem}{Theorem}
\newtheorem{proposition}{Proposition}
\newtheorem{lemma}{Lemma}
\newtheorem{definition}{Definition}
\newtheorem*{remark}{Remark}
\newcommand{\blkdiag}{\mathrm{blkdiag}}

\renewcommand{\Re}{\mathrm{Re}}

\newcommand{\eff}{\textit{eff}}
\usepackage{algorithm}
\usepackage{algorithmicx}
\title{Error Analysis of Randomized Symplectic Model Order Reduction for Hamiltonian systems}
\author{R. Herkert, P. Buchfink, B. Haasdonk, J. Rettberg, J. Fehr}
\begin{document}
\maketitle
\begin{abstract}
Solving high-dimensional dynamical systems in multi-query or real-time applications requires efficient surrogate modelling techniques, as e.g.,\ achieved via model order reduction (MOR). If these systems are Hamiltonian systems their physical structure should be preserved during the reduction, which can be ensured by applying symplectic basis generation techniques such as the complex SVD (cSVD). Recently, randomized symplectic methods such as the randomized complex singular value decomposition (rcSVD) have been developed for a more efficient computation of symplectic bases that preserve the Hamiltonian structure during MOR. In the current paper, we present two error bounds for the rcSVD basis depending on the choice of hyperparameters and show that with a proper choice of hyperparameters, the projection error of rcSVD is at most a constant factor worse than the projection error of cSVD. We provide numerical experiments that demonstrate the efficiency of randomized symplectic basis generation and compare the bounds numerically.
\end{abstract}



\section{Introduction}
On the one hand, classical simulation methods rely on simulation models based on physical principles. On the other hand, data-based modelling techniques using machine learning are becoming increasingly popular. Current trends tend to merge those principles by enriching the physics-based models with data and to include physical prior knowledge in data-based models. In the context of model order reduction (MOR) such a fusion of physics and data-based modelling can be realized by snapshot-based (physical) structure-preserving MOR.
One way to model physical systems while guaranteeing conservation principles,  is using the framework of Hamiltonian systems which are for example often used in mechanics, optics, quantum mechanics or theoretical chemistry. 
The mathematical structure of this kind of systems ensures conservation of the Hamiltonian (which can be understood as the energy contained in the system) and under certain assumptions stability properties \cite{Meyer2017}. These simulation models may be of large scale especially in real-world applications as they may arise from spatially discretized PDEs. Therefore, in multi-query or real-time applications efficient surrogate modelling techniques, e.g.,\ achieved via MOR are required. 
However, classical data-based MOR via the Proper Orthogonal Decomposition (POD) \cite{Volkwein2013} does not necessarily preserve the Hamiltonian structure in the reduced order model (ROM) which could lead to unphysical models that may violate conservation properties and could become unstable. Therefore, it is necessary to ensure the preservation of the Hamiltonian structure by the applied MOR technique.
This can be accomplished by symplectic MOR where the system is projected to a low-dimensional, symplectic subspace \cite{Buchfink2019,MaboudiAfkham2017,Peng2016}. For low-dimensional problems, a symplectic matrix can be computed by numerically solving the proper symplectic decomposition (PSD) optimization problem \cite{Peng2016}. For high-dimensional problems numerically solving the optimization problem is not feasible and other techniques have to be used to construct a reduced basis. A popular method to compute a symplectic basis is the complex singular value decomposition (cSVD) \cite{Peng2016}. This technique involves computing a low-rank matrix approximation, which can also result in high computational costs in the offline-phase. Randomized approaches for computing low-rank matrix factorization \cite{Halko2011,Mahoney2011,Woodruff2014,Murray2023} are a promising way to lower this computational effort while preserving a high approximation quality compared to classical methods. Randomized techniques can be used to solve various numerical linear algebra problems more efficiently, such as the computation of a determinant \cite{Boutsidis2017}, Gram--Schmidt orthonormalization  \cite{Balabanov2022GS}, the computation of an eigenvalue decomposition or an SVD \cite{Halko2011}, rank estimation \cite{Meier2024}, the computation of a LU decomposition \cite{Li2020}, or the computation of a generalized LU decomposition \cite{Demmel2023}. In the context of MOR the capability of randomized algorithms has been shown by applying randomization for more efficient basis generation
\cite{Alla2019,Bach2019,Hochman2014}. In \cite{Buhr2018}
the concept of randomized basis generation is merged with ideas from domain decomposition. Random sketching techniques have further been used for computing parameter-dependent preconditioners \cite{Zahm2016} or for approximating a ROM by its random sketch \cite{Balabanov2019,Balabanov2021}. In \cite{Schleuss23}  time-dependent problems are treated by constructing randomized local approximation spaces in time. While none of these approaches guarantees to preserve a Hamiltonian structure, in \cite{Herkert2023rand} we presented randomized techniques for symplectic basis generation and reported initial encouraging numerical experiments. The scope of the current work is to improve the methods presented there and give a theoretical foundation by mathematical error analysis. Our key contributions are: 
\begin{enumerate}
\item We prove that the randomized complex SVD (rcSVD) \cite{Herkert2023rand} is quasi-optimal in the set of symplectic matrices with orthonormal columns.
\item  We present an error bound depending on the hyperparameters which yields a better understanding of the method and better intuition on how to choose the hyperparameters depending on the problem. 
\item We show how the rcSVD algorithm can be reformulated into a version that works only with real matrices.
\end{enumerate}
Our paper is structured as follows: An introduction to structure-preserving MOR is given  in \Cref{SympMOR}. In \Cref{sec:quasiopt}, we prove quasi-optimality for the rcSVD in the set of symplectic bases with orthonormal columns. \Cref{sec:powinfl} analyzes the influence of power iterations and present an error bound depending on the hyperparameters. We present a formulation of the cSVD algorithm based on real numbers in \Cref{sec:interpretability}. In \Cref{Numerics}, we show numerical experiments that demonstrate the computational efficiency of randomized symplectic basis generation and compare the bounds numerically. The work is concluded in \Cref{Conc}. 

\section{Structure-Preserving MOR}
\label{SympMOR}
In this section, we give an introduction to both, classical structure-preserving, symplectic MOR and randomized structure-preserving, symplectic MOR using the randomized complex SVD.
\subsection{Hamiltonian Systems and Symplectic MOR}\label{subs:sympMor}
We start with an overview on
Hamiltonian systems and structure-preserving, symplectic MOR for parametric high-dimensional Hamiltonian systems. For
a more detailed introduction, we refer to \cite{Benner2021,Benner2017,Benner2015} (MOR), \cite{daSilva2008} (symplectic geometry
and Hamiltonian systems) and \cite{Peng2016} (symplectic MOR of Hamiltonian systems).

We assume to be given a \emph{parametric Hamiltonian
(function)} $\Ham(\cdot; \fmu) \in \calC^1(\R^{2N}, \R)$, depending on a parameter vector $\fmu \in \paramDomain$ with parameter set $\paramDomain \subset \R^{p}$ and a parameter-dependent initial value $\fxInit(\fmu)$. Then, the parametric Hamiltonian system reads: For a given time interval $\It = [\tInit, \tEnd]$ and fixed (but arbitrary) parameter vector
$\fmu \in \paramDomain$, find the solution $\fx(\cdot; \fmu) \in \calC^1(\It, \R^{2N})$ of
\begin{equation}\label{eq:ham_sys}
\begin{aligned}
  \ddt \fx(t; \fmu)
&= \JtN \grad[\fx] \Ham(\fx(t; \fmu); \fmu) \qquad \text{for all } t \in \It,\\
  \fx(\tInit; \fmu) &= \fxInit(\fmu).
\end{aligned}
\end{equation}
with 
 \emph{canonical Poisson matrix}
\begin{align*}
  \JtN := \begin{bmatrix}
    \Z{N} & \I{N}\\
    -\I{N} & \Z{N}
  \end{bmatrix} \in \R^{2N \times 2N},
\end{align*}
where $\I{N}, \Z{N} \in \R^{N \times N}$ denote the identity and zero matrix.
In some cases it is convenient to split the solution
$\fx(t; \fmu) = [\fq(t; \fmu); \fp(t; \fmu)]$ in separate coordinates
$\fq(t; \fmu),\; \fp(t; \fmu) \in \R^N$ which are referred to as the generalized position $\fq$ and generalized momentum $\fp$. Note that here and in the following we use MATLAB-style notation for matrix indexing and stacking. 
One important property of a Hamiltonian system is that the solution preserves the Hamiltonian over time, i.e.,
$\ddt \Ham(\fx(t; \fmu); \fmu) = 0$ for all $t \in \It$.

Symplectic MOR \cite{MaboudiAfkham2017,Peng2016} is a projection-based MOR technique to reduce parametric high-dimensional Hamiltonian systems. It essentially consists of constructing a suitable symplectic reduced basis matrix $\fV \in \R^{2N\times2k}$ which is then used for projecting the full-order system to a reduced surrogate.
It results in a ROM that is a low-dimensional
Hamiltonian system with a \emph{reduced Hamiltonian}
$\Hamr(\fxr) := \Ham(\fV \fxr)$.
This is obtained by (i) the ROB matrix $\fV\in \R^{2N\times2k}$ being a \emph{symplectic matrix} i.e.,
\begin{align*}
  \rT\fV \JtN \fV = \Jtk
\end{align*}
and (ii) setting the projection matrix $\fW \in \R^{2N\times2k}$ as the transpose of the so-called \emph{symplectic inverse}
$\si{\fV}\in \R^{2k\times2N}$ of the ROB matrix $\fV$, i.e.,
\begin{align*}
  \rT\fW := \si{\fV} := \Jtk \rT\fV \TJtN.
\end{align*}
Then, the reduced parametric Hamiltonian system reads: For a fixed (but arbitrary) parameter vector
$\fmu \in \paramDomain$, find the solution $\fxr(\cdot; \fmu) \in \calC^1(\It, \R^{2k})$
of
\begin{equation}\label{eq:ham_sys}
\begin{aligned}
  \ddt \fxr(t; \fmu)
&= \Jtk \grad[\fxr] \Hamr(\fxr(t; \fmu); \fmu) \qquad \text{for all } t \in \It,\\
  \fxr(\tInit; \fmu) &= \si{\fV}\fxInit(\fmu).
\end{aligned}
\end{equation}

\subsection{Symplectic basis generation using the complex SVD (cSVD)}\label{sec:cSVD}
In this section, we present the symplectic basis generation technique complex SVD (cSVD).   
Consider the snapshot matrix $\Xs := [\xs_1,..,\xs_\ns] \in \R^{2N \times \ns}$ with $\xs_i \in \solMnf, i = 1,...,\ns$ where $\solMnf$ denotes the set of all solutions 
\begin{align*}
 \solMnf := \left\{
   \fx(t; \fmu) \,\vert  \, (t, \fmu) \in \It \times \paramDomain
 \right\} \subset \R^{2N}.
\end{align*}
Next, $\Xs$ is split into $\Xs = [\Qs; \Ps]$, with $\Qs,  \Ps \in \R^{N \times \ns}$. The main idea of the cSVD algorithm is to compute a truncated SVD of the complex snapshot matrix $\Uc \fSigma_c\rH{\Vc} \approx \Xc := \Qs + \im \Ps \in \Cn^{N \times \ns}$. Then, the matrix $\Uc \in \Cn^{N\times k}$ is split into real and imaginary part $\Uc = \VQ + \im\VP,$ with $\VQ,\VP\in \R^{N \times k}$ and mapped to
$$
\VcSVD := \mathcal{A}(\Uc) :=
\begin{pmatrix}
 &\VQ &-\VP \\
&\VP &\VQ
\end{pmatrix} \in \R^{2N \times 2k}
.
$$
With this mapping $\mathcal{A}$, a
complex matrix with orthonormal columns is mapped from the complex Stiefel manifold $V_k(\Cn^N)$ to a real symplectic matrix in $\R^{2N\times 2k}$, i.e., $\mathcal{A}: V_k(\Cn^N) \to \R^{2N \times 2k}$ \cite{Peng2016}.
The symplectic cSVD basis matrix $\VcSVD$ and its symplectic inverse $\si{\VcSVD}$ are then used to construct the ROM. In \cite{Buchfink2019} it has been shown that the cSVD procedure yields the optimal symplectic basis in the set of ortho-symplectic matrices (i.e., symplectic with orthonormal columns). Furthermore, every ortho-symplectic matrix $\VE\in \R^{2N \times 2k}$ has the block structure $\VE = [\fE, \JtN\fE]$ where $\fE \in V_k(\R^N)$. Therefore, general ortho-symplectic matrices will in the following be denoted with $\VE.$  
The procedure is summarized as \Cref{alg:cSVD}.

\begin{algorithm}[H]
\caption{Complex SVD (cSVD)}
\label{alg:cSVD}

\hspace*{\algorithmicindent}\textbf{Input}: Snapshot matrix $\Xs \in \R^{2N\times\ns}$ , target size $2k \in \N$ of the ROB, \\ 
\hspace*{\algorithmicindent}\textbf{Output}: Symplectic ROB matrix $\fV_\mathrm{cSVD} \in \R^{2N \times2k}$
\begin{algorithmic}[1]
\item $\Xc = \Xs(1 : N, :) + \im \Xs((N + 1) : (2N ), :)$ \Comment{complex snapshot matrix}
\item $[\Uc, \fSigmac, \Vc] =$ SVD$(\Xc)$ \Comment{basis for approximation of $\Xc$}
\item $\fU_{c(k)} = \Uc(:, 1:k)$ \Comment{truncate to rank-$k$ basis}
\item $\VQ =$ Re$(\fU_{c(k)}) , \VP =$ Im$(\fU_{c(k)})$ \Comment{split in real and imaginary part}
\item $\VcSVD = [\fV_Q, -\fV_P; \fV_P, \fV_Q]$ \Comment{map to symplectic matrix}
\end{algorithmic}
\end{algorithm}
\subsection{Symplectic basis generation using the randomized complex SVD}
In this section, we present a brief summary on randomized matrix factorizations and a refined version of the randomized complex SVD (rcSVD) algorithm from \cite{Herkert2023rand}. In the following we focus on the randomized SVD. Similar techniques can be applied to other types of factorizations. We refer to \cite{Halko2011} for a more detailed presentation on randomized matrix factorization. In this section, we use general notation for the matrix sizes $m,n$ as the results are more general than only covering our case from the previous section. In the context of Hamiltonian systems, we later will use $m = N$ and $n = n_s$.
The computation of a randomized SVD of a matrix $\fA \in \Cn^{m\times n}$ proceeds in two stages. First, using random sampling methods \cite{Halko2011}, a matrix $\fQ \in \Cn^{m\times k},$ with $k \leq m$ and $k \leq n$ with orthonormal columns is computed that approximates $\fA \approx \fQ\rH{\fQ}\fA$. To obtain this, a so-called random sketch $\fY = \fA\fOmega \in \Cn^{m\times k}$ is computed where the sketching matrix $\fOmega \in \Cn^{n\times k}$ is drawn from some random distribution (e.g. an elementwise normal distribution). Then, the columns of $\fY$ are orthonormalized to form the matrix $\fQ$. Based on the approximation of $\fA$ by $\fQ\rH{\fQ}\fA$, a randomized version of the SVD can be formulated: With the definition $\fB:=\rH{\fQ}\fA$, its SVD is $\fB= \fU_\fB \fSigma_\fB \rH{\fV_\fB}$ and by setting $\fU := \fQ \fU_\fB$ we get the randomized SVD $\fA \approx \fU \fSigma_\fB \rH{\fV}_\fB$. The fact that $\fU$ has orthonormal columns follows by its definition as a product of two of such matrices. Instead of using a sketching matrix of target rank $k$, it is known that the approximation quality can be improved by introducing an oversampling parameter $\povs$ and aiming for
$l:=k+\povs$ columns for $\fOmega$ \cite{Halko2011} (see step 2 in \Cref{alg:rCSVD}), and truncate to a rank-$k$ basis (see steps 5, 6, 7 in \Cref{alg:rCSVD}). The method can be further improved by applying power iterations. This means that for $\qpow \in \N_0$, the random sketch is computed as $\fY = \fA(\rH{\fA}\fA)^{\qpow} \fOmega$. Especially for matrices whose singular values decay slowly this can be useful. 
A computational advantage in comparison with a direct factorization of $\fB$ will be achievable if $l\ll n$ and $l \ll m$. This procedure is particularly efficient when using a special random sketching matrix such as the subsampled randomized Fourier transform (SRFT) \cite{Halko2011} that allows the multiplication $\fB \fOmega$ to be performed in $\mathcal{O}(m n \log(l))$ flops. 
\begin{definition}\label{def:SRFT}
An SRFT is an $n \times l$ matrix of the form
$$\fOmega = \sqrt{\frac{n}{l}}\fD\fF \fR, $$
with
\begin{enumerate}
\item $\fD\in \Cn^{n\times n}$ diagonal, with diagonal entries that are independent random variables uniformly distributed on the complex unit circle,
\item  $\fF \in \Cn^{n\times n}$ a unitary discrete Fourier transform (DFT) and
\item $\fR$ $\in \R^{n \times l}$ selection matrix where its columns are drawn randomly without replacement from the columns of the identity matrix  $\I{n}.$
\end{enumerate}
\end{definition}
In order to apply randomization to symplectic basis generation, the idea of the rcSVD algorithm is to replace the computation of the truncated SVD of $\Xc$ with a randomized rank-$k$ approximation of the complex snapshot matrix $\Xc$. The procedure is summarized as \Cref{alg:rCSVD}. In comparison with the original algorithm in \cite{Herkert2023rand}, the truncation step is refined via the computation of an additional SVD of a small matrix (see \Cref{alg:rCSVD} steps 5, 6, 7). This improves the approximation quality and is also necessary for the mathematical analysis presented in the next section which does not work for the original version of the method. 

\begin{algorithm}[H]
\caption{Randomized Complex SVD (rcSVD)}
\label{alg:rCSVD}

\hspace*{\algorithmicindent}\textbf{Input}: Snapshot matrix $\Xs \in \R^{2N\times\ns}$ , target rank $2k \in \N$ of the ROB, \\ 
\hspace*{\algorithmicindent}oversampling parameter $p_\mathrm{ovs} \in \N_0$,  power iteration number $q_\mathrm{pow} \in \N_0$ \\
\hspace*{\algorithmicindent}\textbf{Output}: Symplectic ROB matrix $\fV_\mathrm{rcSVD} \in \R^{2N \times2k}$
\begin{algorithmic}[1]
\item $\Xc = \Xs(1 : N, :) + \im \Xs((N + 1) : (2N ), :)$ \Comment{complex snapshot matrix}
\item $\fOmega$ = SRFT($\ns, l), \text{ with } l:= k+\povs $ \Comment{draw a random sketching matrix}
\item $\fY = \Xc (\rH {\Xc }\Xc )^{\qpow} \fOmega$
\item $[\UY, \fSigmaY, \VY] =$ SVD$(\fY)$ \Comment{basis for approximation of $\fY$}
\item $\fB = \rH{\UY}\Xc$
\item $[\UB, \fSigmaB, \VB] =$ SVD$(\fB)$ \Comment {basis for approximation of $\fB$}
\item $\Ucr = \UY\UB(:, 1:k)$ \Comment{truncate to rank-$k$ basis}
\item $\VQ =$ Re$(\Ucr) , \VP =$ Im$(\Ucr)$ \Comment{split in real and imaginary part}
\item $\VrcSVD = [\fV_Q, -\fV_P; \fV_P, \fV_Q]$ \Comment{map to symplectic matrix}
\end{algorithmic}
\end{algorithm}

\section{Quasi-optimality for the rcSVD in the set of ortho-symplectic matrices}\label{sec:quasiopt}
In \cite{Buchfink2019} it has been shown that the cSVD algorithm \cite{Peng2016} yields an optimal solution of the PSD in the set of ortho-symplectic bases. I.e., for $\Xs = [\fP; \fQ] \in \R^{2N \times \ns}, \VcSVD \in \R^{2N\times 2k}$ it holds that
\begin{align}
\min\limits_{\VE \in \R^{2N \times 2k} \text{ ortho-symplectic}}\vert\vert\Xs - \VE \rT{\VE}\Xs\vert\vert_F^2  = \vert\vert\Xs - \VcSVD \rT{\VcSVD}\Xs\vert\vert_F^2
\end{align}
with projection error
\begin{align}
\vert\vert\Xs - \VcSVD \rT{\VcSVD}\Xs\vert\vert_F^2 = \sum\limits_{j \geq k+1} \sigma_j^2,
\end{align}
where $\sigma_j,j = 1,..,\ns$ denote the singular values of the complex snapshot matrix 
$\Xc = \fQ + \im \fP.$
In the following, we show that the rcSVD procedure (see \Cref{alg:rCSVD}) is quasi-optimal in the set of ortho-symplectic matrices. Before doing so, we recall some results from \cite{Tropp2011Ana} on structured random matrices. The first one states a bound on the smallest singular value of a matrix resulting from randomly sampling rows from a matrix with orthonormal columns. It is a slight reformulation of \cite[Lemma 3.2]{Tropp2011Ana} and therefore, we omit the proof. 

\begin{lemma}[Row sampling \cite{Tropp2011Ana}]\label{lemma1}
Consider a matrix $\fW\in \Cn^{n \times k}$ with orthonormal columns, and define
the quantity $M := n \max\limits_{j=1,...,n}\vert\vert\rT{\fe_j}  \fW\vert\vert_2^2$ where $\fe_j$ denotes the $j$-th unit vector. For a positive parameter $\alpha$, select the sample size $l$ with $ \alpha M \log(k)\leq l \leq n$. Draw a random subset $T$ of size $l$ from $\{1, 2, . . . , n\}$ and define the matrix $\fR\in\R^{n\times l}$ by stacking the corresponding unit vectors as column vectors (see \Cref{def:SRFT}). Then, for $\delta \in [0,1)$ it holds that
\begin{equation}\label{singbound}
\sqrt{\frac{(1-\delta)l}{n}}\leq \sigma_k(\rT{\fR}\fW)
\end{equation}
with failure probability at most 
$$\mathcal{P}_\text{f} = k\left[\frac{e^{-\delta} }{(1-\delta)^{1-\delta}} \right]^{\alpha \log( k)},$$
i.e., \Cref{singbound} holds with probability at least $1 - \mathcal{P}_\text{f}.$
\end{lemma}
Compared to Lemma 3.2 from \cite{Tropp2011Ana}, we removed the bound on the largest singular value (which will not be needed for proving the error bounds/quasi-optimality) to improve the bound for the failure probability.  
The next lemma is a variation of \cite[Lemma 3.4]{Tropp2011Ana}. 
\begin{lemma}[Row norms \cite{Tropp2011Ana}] \label{lemma2} Consider $\fV\in \Cn^{n \times k}$ with orthonormal columns, $\fD \in \Cn^{n\times n}$ diagonal, with diagonal entries that are independent random variables uniformly
distributed on the complex unit circle,
and $\fF \in \Cn^{n\times n}$ a unitary discrete Fourier transform (DFT). Then, $\fF\fD\fV\in \R^{n\times k}$ has orthonormal columns, and for $\beta \geq 1$ it holds that the probability
$$\mathcal{P} \bigg\{\max\limits_{j=1,...,n}\vert\vert\rT{\fe_j}(\fF\fD\fV )\vert\vert_2\geq
\sqrt{\frac{k}{n}}+ \sqrt{\frac{8 \log(\beta n)}{n}}\bigg\}\leq \frac{1}{\beta}.$$ 
\end{lemma}
\begin{proof}
The proof follows identically to the proof of Lemma 3.3 provided in \cite{Tropp2011Ana} because $\fF$ is unitary and $\fD$ is diagonal with diagonal elements which have absolute value 1. 
\end{proof}
By an identical argument as in the proof of \cite[Theorem 3.1]{Tropp2011Ana}, one can show the following probabilistic bounds on the singular values of a matrix with orthonormal columns multiplied by an SRFT. 
\begin{proposition}[The SRFT preserves geometry \cite{Tropp2011Ana}] \label{propSRFT}
Consider $\fV \in \Cn^{n\times k}$ with orthonormal columns. Select a parameter $l$ that satisfies
$$4[\sqrt{k} + \sqrt{8 \log(kn)} ]^2 \log(k) \leq l \leq n.$$
Draw an SRFT matrix $\fOmega\in\R^{n \times l}$. Then, with probability $1-2/k$ it holds that
$$\frac{1}{\sqrt{6}} \leq \sigma_k(\rH{\fOmega}\fV).$$\end{proposition}
\begin{proof}
The proof follows similar to the proof of Theorem 3.2 in \cite{Tropp2011Ana} (by setting $\beta = k$ in \Cref{lemma2} and $\alpha = 4, \delta = 5/6$ in \Cref{lemma1}). However the bound on the failure probability there can be sharpened because the bound on the largest singular value will not be needed to prove the error bounds from \Cref{thmbound,thmadvbound}.
\end{proof}

Lastly, we recall a deterministic error bound \cite[Theorem 9.2]{Halko2011} on the projection error of a randomized rank-$l$ approximation.

\begin{proposition}[Deterministic error bound \cite{Halko2011}]\label{Deterr} 
Consider $\fA \in \Cn^{m \times n}$ with singular value decomposition $\fA = \fU \fSigma \rH{\fV}$, and fix $k \geq 0$. Choose a matrix $\fOmega\in \Cn^{n \times l}$,
and construct the sample matrix $\fY = \fA\fOmega \in \Cn^{m\times l}$. Partition $\fSigma = \blkdiag(\fSigma_1, \fSigma_2)$ with $\fSigma_1 \in \R^{k \times k}, \fSigma_2 \in \R^{(m-k) \times (n-k)}, m \geq k, n \geq k$ and $\fV  = [\fV_1, \fV_2]$ with $\fV_1 \in \Cn^{n \times k}, \fV_2 \in \Cn^{n \times (n-k)}$ and define $\fOmega_1 = \rH{\fV_1}\fOmega$ and $\fOmega_2 = \rH{\fV_2}\fOmega$. Assuming that $\fOmega_1$ has full row rank, the approximation
error satisfies 
$$\vert\vert(\I{m}-\fP_\fY )\fA\vert\vert^2_\mathrm{F} \leq \vert\vert\fSigma_2\vert\vert^2_\mathrm{F} + \vert\vert\fSigma_2\fOmega_2\fOmega_1^\dagger\vert\vert^2_\mathrm{F} \leq \vert\vert\fSigma_2\vert\vert^2_\mathrm{F} + \vert\vert\fSigma_2\vert\vert^2_\mathrm{F}\vert\vert\fOmega_2\vert\vert_2^2\vert\vert\fOmega_1^\dagger\vert\vert_2^2,$$
where $(\cdot)^\dagger$ indicates the pseudo-inverse and  $\fP_\fY$ denotes the orthogonal projector on $\mathrm{range}(\fY)$ i.e., $\fP_\fY = \UY\rH{\UY}$ with $\UY$ the matrix of left singular vectors of $\fY$ corresponding to nonzero singular values.
\end{proposition}
Using these lemmas and propositions, we now prove that the rcSVD procedure yields a basis that is at most a constant factor worse than the optimal cSVD procedure with a constant that is monotonically decreasing in $l$.  

\begin{theorem}\label{thmbound}
If $4(\sqrt{k}+\sqrt{8\log(k\ns)})^2 \log(k) \leq l \leq \ns $, then the rcSVD basis matrix $\VrcSVD \in \R^{2N \times 2k}$ satisfies with failure probability $2/k$
\begin{align}	
\vert\vert\Xs - \VrcSVD \rT{\VrcSVD}\Xs\vert\vert_F^2 \leq C \sum\limits_{j \geq k+1} \sigma_j^2 = C\vert\vert\Xs - \VcSVD \rT{\VcSVD}\Xs\vert\vert_F^2
\end{align} 
with $\sigma_j,j = 1,..,\ns$ the non-increasing sequence of singular values of the complex snapshot matrix 
$\Xc = \fQ + \im \fP,$ $\Xs\in \R^{2N \times \ns}$ and $C = (\sqrt{1 + 6\ns/l}+1)^2$.
\end{theorem}

\begin{proof}
First, we recall from \Cref{sec:cSVD} or \cite{Buchfink2019} that each ortho-symplectic matrix $\VE$ has the structure $\VE = [\fE, \TJtN \fE]$ with $\rT \fE \fE = \I{k}$ and $\rT \fE \JtN \fE = \Z{k}$. Thus, it can be represented as
$\VE = \begin{pmatrix}
\VQ & -\VP \\
\VP & \VQ
\end{pmatrix}$
with $\rT\VQ\VQ + \rT\VP\VP = \I{k}$, $\rT\VP\VQ = \rT\VQ\VP$.
\\

The first step of the proof is to show, that for $\VE$ ortho-symplectic
\begin{align}
\vert\vert\Xs - \VE \rT{\VE}\Xs\vert\vert_F^2 = \vert\vert\Xc - \Uc\rH{\Uc} \Xc \vert\vert_F^2
\end{align}
with $\Uc = \VQ + \im \VP$ and $\Xc = \fQ + \im \fP\in \Cn^{N \times \ns}.$ This can be seen as follows
\begingroup
\allowdisplaybreaks[3]
\begin{align*}
\vert\vert&\Xs - \VE \rT{\VE}\Xs\vert\vert_F^2 = \left|\left|\Xs - \VE
\begin{pmatrix}
\rT{\VQ}\fQ + \rT{\VP}\fP \\
-\rT{\VP}\fQ + \rT{\VQ}\fP
\end{pmatrix} 
\right|\right|_F^2 \\ 
&= \left|\left|\Xs - \begin{pmatrix}
\VQ & -\VP \\
\VP & \VQ
\end{pmatrix}
\begin{pmatrix}
\rT{\VQ}\fQ + \rT{\VP}\fP \\
-\rT{\VP}\fQ + \rT{\VQ}\fP
\end{pmatrix} 
\right|\right|_F^2 \\ 
&= \left|\left|\begin{pmatrix}
\fQ \\
\fP
\end{pmatrix} - 
\begin{pmatrix}
\VQ\rT{\VQ}\fQ + \VQ\rT{\VP}\fP +\VP\rT{\VP}\fQ - \VP\rT{\VQ}\fP\\
\VP\rT{\VQ}\fQ + \VP\rT{\VP}\fP-\VQ\rT{\VP}\fQ + \VQ\rT{\VQ}\fP
\end{pmatrix} 
\right|\right|_F^2 \\ 
&= \vert\vert\fQ - (\VQ\rT{\VQ}\fQ + \VQ\rT{\VP}\fP +\VP\rT{\VP}\fQ - \VP\rT{\VQ}\fP)\vert\vert_F^2  \\
&+ \vert\vert\fP - (\VP\rT{\VQ}\fQ + \VP\rT{\VP}\fP-\VQ\rT{\VP}\fQ + \VQ\rT{\VQ}\fP)\vert\vert_F^2 \\ 
&= \vert\vert\fQ - (\VQ\rT{\VQ}\fQ + \VQ\rT{\VP}\fP +\VP\rT{\VP}\fQ - \VP\rT{\VQ}\fP) \\
&+ \im(\fP - (\VP\rT{\VQ}\fQ + \VP\rT{\VP}\fP-\VQ\rT{\VP}\fQ + \VQ\rT{\VQ}\fP)) \vert\vert_F^2 \\ 
&= \vert\vert\Xc - (\VQ + \im \VP)(\rT{\VQ}\fQ + \rT{\VP}\fP + \im(-\rT{\VP}\fQ + \rT{\VQ}\fP)) \vert\vert_F^2 \\ 
&= \vert\vert\Xc - (\VQ + \im \VP)(\rT{\VQ} -\im\rT{\VP})(\fQ + \im \fP) \vert\vert_F^2 \\ 
&= \vert\vert\Xc - \Uc\rH{\Uc} \Xc \vert\vert_F^2.
\end{align*}
\endgroup
If we insert the randomized basis matrix $\UY \in \mathbb{C}^{N \times l}$ from \Cref{alg:rCSVD} for $\Uc$, in order to bound $\vert\vert\Xc - \UY\rH{\UY} \Xc \vert\vert_F^2$ we can make use of the second bound from \Cref{Deterr}. \\

The next step is to bound the increase in the projection error if we truncate to a basis of size $k$ (see \cite[Section 9.4]{Halko2011}). 
Note that this part of the proof works only with the refined method presented in \Cref{alg:rCSVD} and not with the initial version in \cite{Herkert2023rand}. Let $\Ucr$ be the randomized rank-$k$ basis matrix from \Cref{alg:rCSVD}.
First, we split the error using the triangle inequality: 

\begin{align}
\vert\vert\Xc - \Ucr\rH{(\Ucr)}&\Xc\vert\vert_F = \vert\vert\Xc - \UY\rH{\UY}\Xc + \UY\rH{\UY}\Xc -  \Ucr\rH{(\Ucr)}\Xc\vert\vert_F \\
&\leq \vert\vert\Xc - \UY\rH{\UY}\Xc\vert\vert_F + \vert\vert\UY\rH{\UY}\Xc -  \Ucr\rH{(\Ucr)}\Xc\vert\vert_F. \label{eqn:split}
\end{align}
For $\vert\vert\Xc - \UY\rH{\UY}\Xc\vert\vert_F$ the bounds from \Cref{Deterr} which depend on the random sketching approach can be applied. \\

Next, we define $\fB:=\rH{\UY}\Xc$ according to \Cref{alg:rCSVD} with the singular value decomposition $$\fB = {\fU_\fB} \fSigma_\fB \rH{\fV_\fB},$$ 
 and $${\fU_\fB}_{(k)}{\fSigma_\fB}_{(k)}\rH{{\fV_\fB}_{(k)}} \approx \fB$$ the rank-$k$ truncated SVD of $\fB$, 
 where $${\fU_\fB} \in \Cn^{N \times N}, \fSigma_\fB\in \R^{N \times l}, {\fV_\fB}\in\Cn^{l \times l}$$
and 
$${\fU_\fB}_{(k)} \in \Cn^{N \times k}, {\fSigma_\fB}_{(k)} \in \R^{k \times k}, {\fV_\fB}_{(k)} \in \Cn^{l \times k}.$$
 Using these quantities, we set $$\fC:= \UY \fB  = \UY{\fU_\fB} \fSigma_\fB \rH{\fV_\fB}.$$ Then, $\UY{\fU_\fB} \in \R^{N  \times l}$ has orthonormal columns. Thus, a singular value decomposition $\fC = {\fU_\fC} \fSigma_\fC \rH{\fV_\fC}$ can be formed by constructing $\fU_\fC \in \Cn^{N \times N}$ from extending $\fU_\fY\fU_\fB$ by orthogonal columns, $\fSigma_\fC\in \R^{N \times l} $ from zero padding of $\fSigma_\fB$ and $\rH{\fV_\fC}$ equals $\rH{\fV_\fB}$.
As $\Ucr =\UY{\fU_\fB}_{(k)}$ consists of the first $k$ columns of ${\fU_\fC}$, it follows that $\Ucr\rH{(\Ucr)} \fC=\Ucr {\fSigma_\fB}_{(k)} \rH{{\fV_\fB}_{(k)}} \approx \fC$ is the rank-$k$ truncated SVD of $\fC$.
Furthermore, we have
\begin{align*}
\Ucr\rH{(\Ucr)} \fC&=\Ucr\rH{(\Ucr)} \UY\fB= \Ucr\rH{(\Ucr)}\UY\rH{\UY}\Xc \\
&= \underbrace{\UY{\fU_\fB}_{(k)}}_{= \Ucr}\underbrace{\rH{\fU}_{\fB_{(k)}}\rH{\UY}\UY\rH{\UY}}_{={\rH{\fU}_\fB}_{(k)}\rH{\UY}= \rH {{(\Ucr)}}}\Xc =\Ucr\rH{(\Ucr)}\Xc.
\end{align*}
Let $\Xc_{(k)} = \fU_{\Xc(k)} \fSigma_{\Xc(k)}\rH{\fV_{\Xc(k)}}$ denote the rank-$k$ truncated SVD of $\Xc$. Since the matrix
$\UY\rH{\UY} \fU_{\Xc(k)} \fSigma_{\Xc(k)} \rH{\fV_{\Xc(k)}}$ has at most rank $k$ 
\begin{align}
\vert\vert\underbrace{\UY\rH{\UY}\Xc}_{=\fC} -  &\underbrace{\Ucr\rH{(\Ucr)}\Xc}_{\Ucr\rH{(\Ucr)} \fC}\vert\vert_F= \vert\vert\fC -  \Ucr\rH{(\Ucr)} \fC\vert\vert_F  \\
&\leq\vert\vert\fC -  \UY\rH{\UY} \fU_{\Xc(k)} \fSigma_{\Xc(k)} \rH{\fV_{\Xc(k)}}\vert\vert_F\label{eqn:truncbound1} \\
&= \vert\vert\UY\rH{\UY}\Xc -  \UY\rH{\UY} \fU_{\Xc(k)} \fSigma_{\Xc(k)} \rH{\fV_{\Xc(k)}}\vert\vert_F\\
&= \vert\vert\UY\rH{\UY}(\Xc -  \fU_{\Xc(k)} \fSigma_{\Xc(k)} \rH {\fV}_{\Xc(k)})\vert\vert_F \label{eqn:truncbound2}\\
& \leq \vert\vert\Xc -  \fU_{\Xc(k)} \fSigma_{\Xc(k)} \rH{\fV_{\Xc(k)}}\vert\vert_F \label{eqn:truncbound3}\\
&= \sqrt{\sum\limits_{j \geq k+1} \sigma_j^2}, \label{eqn:truncbound} 
\end{align}
where we used the best-approximation property for obtaining (\ref{eqn:truncbound1}), factoring out for (\ref{eqn:truncbound2}), non-expansiveness of the orthogonal projection $\fU_\fY\rH{\fU}_\fY$ for obtaining (\ref{eqn:truncbound3}) and the definition of the singular values to reach (\ref{eqn:truncbound}).   

Together with \Cref{eqn:split}, \Cref{eqn:truncbound}, \Cref{Deterr} and \Cref{propSRFT}  the above implies that with failure probability $2/k$

\begin{align*} \label{quasi_opt}
\vert\vert\Xs - \VrcSVD &\rT{\VrcSVD}\Xs\vert\vert_F = \vert\vert\Xc - \Ucr\rH{(\Ucr)} \Xc \vert\vert_F \\
&\overset{(\ref{eqn:split})}{\leq} \vert\vert\Xc - \UY\rH{\UY}\Xc\vert\vert_F + \vert\vert\UY\rH{\UY}\Xc -  \Ucr\rH{(\Ucr)}\Xc\vert\vert_F
\\
&\overset{(\ref{eqn:truncbound})}{\leq} \vert\vert\Xc - \UY\rH{\UY}\Xc\vert\vert_F + \sqrt{\sum\limits_{j \geq k+1} \sigma_j^2}\\ 
&\overset{\text{Prop}.\ \ref{Deterr}}{\leq}  \sqrt{\vert\vert\fSigma_2\vert\vert_F^2 + \vert\vert\fSigma_2\vert\vert_F^2\vert\vert\fOmega_2\vert\vert_2^2\vert\vert\fOmega_1^\dagger\vert\vert_2^2} + \sqrt{\sum\limits_{j \geq k+1} \sigma_j^2}\\
&\overset{\text{Prop}.\ \ref{propSRFT}}{\leq} \sqrt{\vert\vert\fSigma_2\vert\vert_F^2 + 6\ns/l\vert\vert\fSigma_2\vert\vert_F^2} + \sqrt{\sum\limits_{j \geq k+1} \sigma_j^2}\\
&= (\sqrt{1 +6\ns/l} + 1) \sqrt{\sum\limits_{j \geq k+1} \sigma_j^2} 
\end{align*}

if $4(\sqrt{k}+\sqrt{8\log(k\ns)})^2 \log(k) \leq l \leq \ns.$ Here, in the second last inequality we bound 
$$\vert\vert\fOmega_2\vert\vert_2^2 = \vert\vert\rH{\fV_2}\fOmega\vert\vert_2^2  \leq \vert\vert\rH{\fV_2}\vert\vert_2^2\vert\vert\fOmega\vert\vert_2^2 = \vert\vert\fOmega\vert\vert_2^2\leq \ns/l$$ 
as, up to the scaling factor $\sqrt{\ns/l}$, the SFRT matrix $\fOmega =\sqrt{\ns/l}\fD\fF\fR$ (\Cref{def:SRFT}) has orthonormal columns, i.e., $\vert\vert\fD\fF\fR\vert\vert_2 = 1$. Moreover, we bound $\vert\vert\fOmega_1^\dagger\vert\vert_2$ via
$$\vert\vert\fOmega_1^\dagger\vert\vert_2 = \sigma_1(\fOmega_1^\dagger) = \sigma_k (\fOmega_1)^{-1} = \sigma_k (\rH{\fOmega_1})^{-1}  = \sigma_k (\rH{\fOmega}\fV_1)^{-1} \overset{\text{Prop}.\ \ref{propSRFT}}{\leq} \sqrt{6}.$$

\end{proof}
From this bound we obtain a better understanding of the method: We know that we are at most a factor $(\sqrt{1 +6\ns/l} + 1)$ worse than the optimal solution (in the set of ortho-symplectic matrices) which gives the method a stronger theoretical foundation. Furthermore, it tells us how to choose hyperparameters to obtain theoretical guarantees: Given number of snapshots $\ns$ and a target rank $k$ choose $\povs$ such that $$4(\sqrt{k}+\sqrt{8\log(k\ns)})^2 \log(k) - k\leq \povs \leq \ns-k.$$ 

\section{Influence of Power Iterations on the Error (Bound)}\label{sec:powinfl}
In this section, we analyze how the choice of the number of power iterations $\qpow$ influences the error (bound) in interplay with the oversampling parameter $\povs.$ First, we reformulate Theorem 4.4 from \cite{Gu2015} for complex matrices: 
\begin{proposition}[\cite{Gu2015}]\label{propadvdetbound}
Consider $\fA\in \Cn^{m\times n}$, $\mathrm{rank}(\fA) = \min(m,n)$ and $\fQ$ a matrix with orthonormal columns spanning the range of $\fY = \fA(\rH{\fA}\fA)^{\qpow}\fOmega\in \Cn^{m\times l}, l = k +\povs \leq n$ with a sketching matrix $ \fOmega\in \Cn^{n \times l}$. Let $\mathrm{rank}(\fY$) = $l$.
Then, for any $s$ with $0\leq s\leq l-k$ holds \footnote{Note that due to preventing a notation clash we renamed the parameter $p$ from \cite{Gu2015} to $s.$} 
\begin{align}
\vert\vert(\I{m} - \fQ\rH{\fQ})\fA\vert\vert_F &\leq \vert\vert\fA - \fQ\fB_{(k)}\vert\vert_F \\
&\leq \sqrt{ \frac{\alpha^2\vert\vert\fOmega_2\vert\vert_2^2\vert\vert\fOmega_1^\dagger\vert\vert_2^2}{1+\gamma^2\vert\vert\fOmega_2\vert\vert_2^2\vert\vert\fOmega_1^\dagger\vert\vert_2^2}  +\sum\limits_{j \geq k+1} \sigma_j^2}\\
&\leq \sqrt{\alpha^2\vert\vert\fOmega_2\vert\vert_2^2\vert\vert\fOmega_1^\dagger\vert\vert_2^2+\sum\limits_{j\geq k+1} \sigma_j^2} 
\end{align}
with  $\fB_{(k)}$ the rank-$k$ truncated SVD of $\rH \fQ\fA$ and $\alpha = \sqrt k \sigma_{l-s+1}(\frac{\sigma_{l-s+1}}{\sigma_{k}})^{2\qpow},$   $\gamma = \frac{\sigma_{l-s+1}}{\sigma_1}(\frac{\sigma_{l-s+1}}{\sigma_{k}})^{2\qpow}$ where  $\sigma_j,j = 1,..,\ns$ denote the non-increasing sequence of singular values of $\fA$.
The matrices $\fOmega_1, \fOmega_2$ are defined as follows: Let $\hat \fOmega := \rT \fV\fOmega$ and split $\hat \fOmega = \begin{pmatrix}
\fOmega_1\\
\fOmega_2
\end{pmatrix}$ with $\fOmega_1  \in \Cn^{(l-s)\times l}$, $\fOmega_2  \in \Cn^{(n-l+s)\times l}$ where we assume that $\fOmega_1$ has full row rank.
\end{proposition}
In \cite{Gu2015} only real matrices with $m \geq n$ have been assumed. However, the proof works in a similar way also in a more general setting for complex matrices of arbitrary size as we will explain in the following. The additional assumption $\mathrm{rank}(\fY$) = $l$ is very likely to be fulfilled in practice because the orthogonal complement $(\rH{\fA})^\perp$ of $\rH{\fA}$ is a null set in $\R^n$ and it is very unlikely (zero probability in exact arithmetics) that for a random vector $\fomega \in \R^{n}$ it holds that $\fomega \in (\rH{\fA})^\perp$. Furthermore, a family of random vectors $\fomega_i \in \R^{n}, i = 1,..,l\leq n$ is linearly dependent with zero probability in exact arithmetics. \\ 
In order to prove \Cref{propadvdetbound}, the following inequality is proven first:
\begin{lemma}\label{lemma3}
Consider $\fA\in \Cn^{m\times n}, \fQ\in \Cn^{m\times k}$ with orthonormal columns.  Then, 
\begin{align*}
\vert\vert(\I{m} - \fQ\rH{\fQ})\fA\vert\vert_F  &\overset{(\mathrm{a})}{\leq} \vert\vert\fA-\fQ\fB_{(k)}\vert\vert_F\overset{(\mathrm{b})}{\leq}\vert\vert\fA-\fQ\rH{\fQ}\fA_{(k)}\vert\vert_F  \\
&\overset{(\mathrm{c})}{\leq} \sqrt{\vert\vert(\I{m}-\fQ\rH{\fQ}) \fA_{(k)} \vert\vert_F^2 + \sum\limits_{j \geq k+1} \sigma_j^2}.
\end{align*}
with $\fA_{(k)}$ the rank-$k$ truncated SVD of $\fA$.
\end{lemma}
\begin{proof}
We start with the inequality (a): 
\begin{align}
&\vert\vert\fA- \fQ\fB_{(k)}\vert\vert_F^2 = \vert\vert\fA-\fQ\rH \fQ\fA + \fQ\rH \fQ \fA -  \fQ\fB_{(k)}\vert\vert_F^2 \notag \\
&= \vert\vert(\I{m} -\fQ\rH{\fQ}) \fA\vert\vert_F^2 +2\Re(\tr(\rH{\fA}\underbrace{(\I{m} - \fQ\rH{\fQ})\fQ}_{=0}( \rH{\fQ} \fA - \fB_{(k)}))) \\ &+\vert\vert \fQ(\rH{\fQ} \fA - \fB_{(k)})\vert\vert_F^2 \notag\\
&= \vert\vert(\I{m} -\fQ\rH{\fQ}) \fA\vert\vert_F^2 +\vert\vert \rH{\fQ} \fA - \fB_{(k)}\vert\vert_F^2 \label{eqnAQB}
\end{align}
which implies
\begin{align*}
\vert\vert(\I{m} -\fQ\rH{\fQ}) \fA\vert\vert_F^2 = \vert\vert\fA- \fQ\fB_{(k)}\vert\vert_F^2 - \vert\vert\rH{\fQ} \fA - \fB_{(k)}\vert\vert_F^2 \leq \vert\vert\fA- \fQ\fB_{(k)}\vert\vert_F^2.
\end{align*}
Next, inequality (b) is shown: 
By the identical argument as in \Cref{eqnAQB} for every $\fB\in \Cn^{k\times n}$ it holds that $$\vert\vert\fA-\fQ\fB\vert\vert_F^2 = \vert\vert(\I{m} -\fQ\rH{\fQ}) \fA\vert\vert_F^2 +\vert\vert \rH{\fQ} \fA - \fB\vert\vert_F^2.$$
A rank-$k$-minimizer of $\vert\vert\rH{\fQ} \fA -\fB\vert\vert_F$ is known to be the rank-$k$ truncated SVD of $\rH\fQ\fA\approx \fB_{(k)}.$
 Since $\vert\vert(\I{m} -\fQ\rH{\fQ}) \fA\vert\vert_F^2$ does not depend on $\fB$, a rank-$k$ minimizer $\fB$ of $\vert\vert\fA-\fQ\fB\vert\vert_F$ is a rank-$k$-minimizer of $\vert\vert\rH{\fQ} \fA -\fB\vert\vert_F$.
Since $\fQ\rH{\fQ}\fA_{(k)} $ is at most of rank $k$, we obtain (b) 
$$\vert\vert\fA-\fQ\fB_{(k)}\vert\vert_F\leq\vert\vert\fA-\fQ\rH{\fQ}\fA_{(k)}\vert\vert_F.$$
Lastly, inequality (c) is shown: 
\begin{align*}
&\vert\vert\fA-\fQ\rH{\fQ}\fA_{(k)}\vert\vert_F^2 = \vert\vert\fA-\fA_{(k)} + \fA_{(k)} - \fQ\rH{\fQ}\fA_{(k)}\vert\vert_F^2 \\
&= \vert\vert\fA-\fA_{(k)}\vert\vert_F^2 + 2\Re(\tr(\rH{(\fA- \fA_{(k)})}(\fA_{(k)} - \fQ\rH{\fQ}\fA_{(k)}))) \\ 
& +\vert\vert\fA_{(k)} - \fQ\rH{\fQ}\fA_{(k)}\vert\vert_F^2\\
&= \vert\vert\fA-\fA_{(k)}\vert\vert_F^2 + \vert\vert\fA_{(k)} - \fQ\rH{\fQ}\fA_{(k)}\vert\vert_F^2\\
&= \sum\limits_{j \geq k+1} \sigma_j^2 + \vert\vert\fA_{(k)} - \fQ\rH{\fQ}\fA_{(k)}\vert\vert_F^2.
\end{align*}
where the middle term vanishes as
\begin{align*}
2\Re(&\tr(\rH{(\fA- \fA_{(k)})}(\fA_{(k)} - \fQ\rH{\fQ}\fA_{(k)}))) \\ 
&= 2\Re(\tr(\rH{(\fA- \fA_{(k)}}(\I{m} - \fQ\rH{\fQ})\fA_{(k)})))\\
&=  2\Re(\tr((\I{m} - \fQ\rH{\fQ})\underbrace{\fA_{(k)}\rH{(\fA- \fA_{(k)})}}_{=0})) = 0.
\end{align*}
 
\end{proof}

Next, another auxiliary statement is proven:
\begin{lemma}\label{lemma4}
Consider the matrix $\fA\in \Cn^{m\times n}$, with $\mathrm{rank}(\fA)=\min(n,m)$, $\fX \in \Cn^{l\times l}, l < \min(n,m)$ to be non-singular, $\fOmega\in \Cn^{n\times l}$ such that $\mathrm{rank}(\fY) = l$ with $\fY = (\fA\rH \fA)^{\qpow}\fA\fOmega$. Let $\fQ\fR = \fY$ be the QR-factorization of $\fY$ with $\fQ \in \Cn^{m\times l}, \fR \in \Cn^{l\times l}$, with $\hat \fQ\hat \fR = \fY_X:= \fY\fX$ be the QR-factorization of $\fY_X$ with $\hat\fQ \in \Cn^{m\times l}, \hat\fR \in \Cn^{l\times l}$. 
Then, 
\begin{equation}
\label{eqnQQQQ}
\fQ\rH \fQ = \hat \fQ \rH{\hat \fQ}.
\end{equation}
\end{lemma}
\begin{proof}
The proof follows from elementary calculations:
\begin{align*}
\fQ\rH \fQ &= \fQ\fR\fR^{-1}(\rH{\fR})^{-1}\rH \fR \rH \fQ = \fQ\fR(\rH \fR \fR)^{-1}\rH \fR \rH \fQ \\
&=  \fQ\fR(\rH \fR\rH \fQ \fQ\fR)^{-1}\rH \fR \rH \fQ = \fY (\rH \fY \fY)^{-1}\rH \fY\\
&= \fY \fX\fX^{-1}(\rH \fY \fY)^{-1} (\rH \fX)^{-1}\rH \fX\rH \fY \\&= \fY \fX(\rH \fX\rH \fY \fY\fX)^{-1} \rH \fX\rH \fY = \fY_X (\rH {\fY_X} \fY_X)^{-1}\rH{ \fY_X} \\
&= \hat \fQ \rH{\hat \fQ}
\end{align*}
where the last equality follows from the equivalent reverse arguments of the first two lines.
\end{proof}
The next step is proving a further upper bound:  
\begin{lemma}\label{lemma5}
Consider $\fA\in \Cn^{m\times n}$, $\mathrm{rank}(\fA)=\min(n,m)$ with $\fA_{(k)}$ the rank-$k$-best approximation of $\fA$. Let $\fQ\fR = \fY= (\fA\rH \fA)^{\qpow}\fA\fOmega$ be the QR-factorization of $\fY$ and $\mathrm{rank}(\fY) = l$. Then, 
\begin{align}\label{lemmabound}
\vert\vert(\I{m}-\fQ\rH{\fQ}) \fA_{(k)} \vert\vert_F^2 \leq \frac{\alpha^2\vert\vert\fOmega_2\vert\vert_2^2\vert\vert\fOmega_1^\dagger\vert\vert_2^2}{1+\gamma^2\vert\vert\fOmega_2\vert\vert_2^2\vert\vert\fOmega_1^\dagger\vert\vert_2^2}, 
\end{align}
with $\alpha, \gamma,\fOmega_1, \fOmega_2$ defined as in \Cref{propadvdetbound}.
\end{lemma}
\begin{proof}
Let $\fA = \fU \fSigma \rH  \fV, \fU\in \R^{m \times \min(m,n)}, \fSigma\in \R^{\min(m,n) \times n}, \fV\in \R^{n \times n}$ be the singular value decomposition of $\fA$. Define $\hat \fOmega = \rH \fV\fOmega$ and split $\hat \fOmega = \begin{pmatrix}
\fOmega_1\\
\fOmega_2
\end{pmatrix}$, with $\fOmega_1  \in \Cn^{(l-s)\times l}$, $\fOmega_2  \in \Cn^{(n-l+s)\times l}$. Split $\fSigma =\text{blkdiag}(\fSigma_1,\fSigma_2, \fSigma_3)$, with $\fSigma_1\in \R^{k\times k},\fSigma_2\in \R^{(l-s-k)\times (l-s-k)}, \fSigma_3\in \R^{(\min(m,n)-l+s)\times (n-l+s)}).$
Let $\fOmega_1^\dagger$ denote the pseudo-inverse of $\fOmega_1$. Since $\fOmega_1$ has full row rank it holds that $$\fOmega_1\fOmega_1^\dagger = \I{l-s}.$$
Choose $$\fX = \left[ \fOmega_1^\dagger\begin{pmatrix}
\fSigma_1 & \Z{k,l-s-k} \\
\Z{l-s-k,k} &\fSigma_2
\end{pmatrix}^{-(2\qpow+1)}, \hat \fX  \right],$$ where the matrix $\hat \fX \in \Cn^{l\times s}$ is chosen such that $\fX$ is non-singular and $\hat \fOmega_1\hat \fX = 0.$ Since the matrix $\fOmega_1^\dagger\begin{pmatrix}
\fSigma_1 & \Z{k,l-s-k} \\
\Z{l-s-k,k} &\fSigma_2
\end{pmatrix}^{-(2\qpow+1)}$ has full column rank (as $\fOmega_1$ has full row rank and $\sigma_i > 0, i = i...\min(m,n)$) and for $s\geq 1$ such a matrix $\fX$ must always exist because of the rank–nullity theorem as $\hat \fOmega_1\in \Cn^{l-s\times l}$ is assumed to have full row rank, i.e., rank($\hat \fOmega_1$) = $l-s.$ Therefore $$\mathrm{dim} \mathrm{(ker}(\hat \fOmega_1)) = l - (l-s) = s$$ 
must hold. Hence, $s$ linearly independent vectors from ker($\hat \fOmega_1)$ can be chosen and stacked to form $\hat \fX$. The column vectors of $\fOmega_1^\dagger\begin{pmatrix}
\fSigma_1 & \Z{k,l-s-k} \\
\Z{l-s-k,k} &\fSigma_2
\end{pmatrix}^{-(2\qpow+1)}$ do not lie in the null-space of $\hat \fOmega_1$ by construction and therefore must be linearly independent from $\hat \fX$. 
With the choice $$\fX = \left[ \fOmega_1^\dagger\begin{pmatrix}
\fSigma_1 & \Z{k,l-s-k} \\
\Z{l-s-k,k} &\fSigma_2
\end{pmatrix}^{-(2\qpow+1)}, \hat \fX  \right]$$ we have the following structure for $\fY\fX$: 
$$\fY\fX = \fU\begin{pmatrix}
\I{k} &\Z{k,l-s-k} &\Z{k,n-l+s} \\
\Z{l-s-k,k} &\I{l-s-k} & \Z{l-s-k,n-l+s} \\
\fH_1 &\fH_2 &\fH_3 
\end{pmatrix}$$
with 
$$\fH_1 = \fSigma_3(\rT{\fSigma_3}\fSigma_3)^{2\qpow}\fOmega_2\fOmega_1^\dagger\begin{pmatrix}
\fSigma_1^{-2\qpow+1} \\ \Z{l-s-k,k}
\end{pmatrix}\in \R^{(\min(m,n)-l+s)\times k}$$ 
$$\fH_2 = \fSigma_3(\rT{\fSigma_3}\fSigma_3)^{2\qpow}\fOmega_2\fOmega_1^\dagger\begin{pmatrix}
\Z{k,l-s-k} \\ \fSigma_2^{-2\qpow+1} 
\end{pmatrix}\in \R^{(\min(m,n)-l+s)\times (l-s-k)}$$
$$\fH_3 = \fSigma_3(\rT{\fSigma_3}\fSigma_3)^{2\qpow}\fOmega_2\hat \fX \in \R^{(\min(m,n)-l+s)\times (n-l+s)}.$$
Next, we similarly split the QR-factorization of $\fY\fX$ into blocks:
$$\fY\fX = \hat \fQ\hat \fR = ({\hat \fQ}_1 \ {\hat \fQ}_2 \ {\hat \fQ}_3)\begin{pmatrix}
\hat{\fR}_{11} &\hat{\fR}_{12} &\hat{\fR}_{13} \\ 
\Z{} &\hat{\fR}_{22} &\hat{\fR}_{23} \\
\Z{} & \Z{} &\hat{\fR}_{33}
\end{pmatrix}. 
$$
This also results in the QR factorization
\begin{equation}
\label{eqnsmallQR}
\fU\begin{pmatrix}
\I{k} \\ \Z{l-s-k,k} \\ \fH_1
\end{pmatrix}  = \hat{\fQ}_1 \hat{\fR}_{11}.
\end{equation}
With (\ref{eqnQQQQ}) and by restricting the number of columns of $\hat \fQ$ we have
\begin{align}\label{eqn:restrictcolumns}
\vert\vert(\I{m} - \fQ\rH \fQ) \fA_{(k)}\vert\vert = \vert\vert(\I{m} - \hat{\fQ}\rH{\hat{\fQ}}) \fA_{(k)}\vert\vert\leq \vert\vert(\I{m} - \hat{\fQ}_1\rH{\hat{\fQ}_1}) \fA_{(k)}\vert\vert
\end{align}
with $\fA_{(k)}  = \fU \blkdiag(\fSigma_1, \Z{l-s-k}, \Z{n-l+s}) \rH \fV.$
The last step of the proof is to derive a bound for $\vert\vert(\I{m} - \hat{\fQ}_1\rH{\hat{\fQ}_1}) \fA_{(k)}\vert\vert$. 
First, we use the QR factorization from (\ref{eqnsmallQR}) to reformulate $\vert\vert(\I{m} - \hat{\fQ}_1\rH{\hat{\fQ}_1}) \fA_{(k)}\vert\vert$ :
\begin{align}\notag
&\vert\vert(\I{m} - \hat{\fQ}_1\rH{\hat{\fQ}_1}) \fA_{(k)}\vert\vert_F = \\&\left|\left|\left(\I{\min(m,n)} - \begin{pmatrix}\I{k} \\ \Z{l-s-k,k} \\\fH_1 \end{pmatrix} \hat{\fR}_{11}^{-1}(\rH{\hat{\fR}_{11}})^{-1}\rH{\begin{pmatrix}\I{k} \\ \Z{l-s-k,k} \\\fH_1 \end{pmatrix}}\right) \begin{pmatrix} \fSigma_1 \\  \Z{l-s-l,k} \\ \Z{min(m,n)-l+s,k} \end{pmatrix}\right|\right|_F. \label{boundQ1Q1A}
\end{align}

Next, we reformulate $\hat{\fR}_{11}^{-1}(\rH{\hat{\fR}_{11}})^{-1}$: 
\begin{align} 
\hat{\fR}_{11}^{-1}(\rH{\hat{\fR}_{11}})^{-1} = (\rH{\hat \fR}_{11}\hat{\fR}_{11})^{-1} &= \left(\rH{\begin{pmatrix}\I{k} \\ \Z{l-s-k,k} \\\fH_1 \end{pmatrix}} \rH{\hat{\fQ}_1} \hat{\fQ}_1 \begin{pmatrix}\I{k} \\ \Z{l-s-k,k} \\ \fH_1 \end{pmatrix}
\right)^{-1} \\& = (\I{k}   + \rH{\fH_1} \fH_1)^{-1}. \label{eqn:refR}
\end{align}
Inserting that into \Cref{boundQ1Q1A} results in 
\begin{align}
&\vert\vert(\I{m} - \hat{\fQ}_1\rH{\hat{\fQ}_1}) \fA_{(k)}\vert\vert_F\\
&=\left|\left|\left(\I{\min(m,n)} - \begin{pmatrix}\I{k} \\ \Z{l-s-k,k} \\\fH_1 \end{pmatrix} \hat{\fR}_{11}^{-1}(\rH{\hat{\fR}_{11}})^{-1}\rH{\begin{pmatrix}\I{k} \\ \Z{l-s-k,k} \\\fH_1 \end{pmatrix}}\right)\begin{pmatrix}\fSigma_1 \\  \Z{l-s-k,k} \\ \Z{\min(m,n)-l+s,k} \end{pmatrix}\right|\right|_F \notag\\ 
&\overset{\ref{eqn:refR}}{=}  \left|\left|\begin{pmatrix}\fSigma_1 \\ \Z{l-s-k,k} \\ \Z{\min(m,n)-l+s,k} \end{pmatrix} - \begin{pmatrix}\I{k} \\ \Z{l-s-k,k} \\\fH_1 \end{pmatrix} (\I{k}   + \rH{\fH_1} \fH_1)^{-1}\fSigma_1 \right|\right|_F \notag\\ 
&= \left|\left|\begin{pmatrix}
\I{k} - (\I{k} + \rH{\fH_1}\fH_1)^{-1} \\ 
\fH_1(\I{k} + \rH{\fH_1}\fH_1)^{-1}
\end{pmatrix}\fSigma_1\right|\right|_F. 
\label{eqnH1Sigma1}
\end{align}

Next, we reformulate the second component $\fH_1(\I{k} + \rH{\fH_1}\fH_1)^{-1}$ of the right side of (\ref{eqnH1Sigma1}): 
\begin{align}
\fH_1(\I{k} + \rH{\fH_1}\fH_1)^{-1} &= (\fH_1^{-1})^{-1}(\I{k} + \rH{\fH_1}\fH_1)^{-1} = ((\I{k} + \rH{\fH_1}\fH_1)\fH_1^{-1})^{-1}\notag\\
&= (\fH_1^{-1} + \rH{\fH_1})^{-1} = (\fH_1^{-1}(\I{n-l+s} + \fH_1\rH{\fH_1}))^{-1} \notag\\
&=  (\I{\min(m,n)-l+s} + \fH_1\rH{\fH_1})^{-1} \fH_1
\label{eqn_2ndcomp}
\end{align}
and then the first component of the right hand side of (\ref{eqnH1Sigma1})
\begin{align*}
(\I{k}- (\I{k} + \rH {\fH_1}\fH_1)^{-1}) &= (\I{k} + \rH {\fH_1}\fH_1)(\I{k} + \rH {\fH_1}\fH_1)^{-1} - (\I{k} + \rH {\fH_1}\fH_1)^{-1} \\
&= (\rH{\fH_1}\fH_1)(\I{k} + \rH {\fH_1}\fH_1)^{-1}\\
& \overset{(\ref{eqn_2ndcomp})}{=} \rH{\fH_1}(\I{\min(m,n)-l+s} + \fH_1\rH {\fH_1})^{-1}\fH_1.
\end{align*}
Inserting this into (\ref{eqnH1Sigma1}) and factorizing yields
\begin{align*}
&\vert\vert(\I{m} - \hat{\fQ}_1\rH{\hat{\fQ}_1}) \fA_{(k)}\vert\vert_F \\ &= \left\vert\left\vert\begin{pmatrix}
\rH{\fH_1} \\ 
 \I{\min(m,n)-(l-s)}
\end{pmatrix}(\I{\min(m,n)-(l-s)} + \fH_1\rH{\fH_1})^{-1}\fH_1 \fSigma_1  \right\vert\right\vert_F  \\
&= \tr(\rH{(\fH_1 \fSigma_1)}(\I{\min(m,n)-(l-s)} + \fH_1\rH{\fH_1})^{-\textsf{H}} \fH_1 \fSigma_1)\\
&= \tr(\rH{(\fH_1 \fSigma_1)}(\I{\min(m,n)-(l-s)} + \fH_1\rH{\fH_1})^{-1} \fH_1 \fSigma_1) \\
&= \tr((\rmH{(\fH_1 \fSigma_1)})^{-1}(\I{\min(m,n)-(l-s)} + \fH_1\rH{\fH_1})^{-1} ((\fH_1 \fSigma_1)^{-1})^{-1}) \\
&= \tr(((\fH_1 \fSigma_1)^{-1}(\I{\min(m,n)-(l-s)} + \fH_1\rH{\fH_1})\rmH{(\fH_1 \fSigma_1)})^{-1})
\\
&= \tr((\fSigma_1^{-1}\fH_1^{-1} (\I{\min(m,n)-(l-s)} + \fH_1\rH{\fH_1})\rmH{\fH_1} \fSigma_1^{-1})^{-1})\\
&= \tr((\fSigma_1^{-1}\fH_1^{-1}\rmH{\fH_1} \fSigma_1^{-1} + \fSigma_1^{-1}\fSigma_1^{-1})^{-1})\\
&= \tr(((\fSigma_1\rH{\fH_1}\fH_1 \fSigma_1)^{-1} + \fSigma_1^{-2})^{-1}). 
\end{align*}
Now, we further analyze 
$\tr(((\fSigma_1\rH{\fH_1}\fH_1 \fSigma_1)^{-1} + \fSigma_1^{-2})^{-1})$.
First, we recall that for Hermitian matrices $\fE, \fF\in \Cn^{k \times k}$ it follows from Courant-Fischer that $$\lambda_j(\fE+\fF)\geq \lambda_{\text{min}}(\fE) +\lambda_j(\fF), \text{ with } 1\leq j\leq k.$$ \\
This implies, that
\begin{align*}
\lambda_j((\fSigma_1\rH{\fH_1}\fH_1 \fSigma_1)^{-1} + \fSigma_1^{-2})&\geq \lambda_{\text{min}}((\fSigma_1\rH{\fH_1}\fH_1 \fSigma_1)^{-1}) +\lambda_j(\fSigma_1^{-2}) \\
&= \vert\vert(\fSigma_1\rH{\fH_1}\fH_1 \fSigma_1\vert\vert_2^{-1} +\lambda_j(\fSigma_1^{-2}) \\
&= \vert\vert\fH_1 \fSigma_1\vert\vert_2^{-2} +\lambda_j(\fSigma_1^{-2}) \\
&= \lambda_j(\vert\vert\fH_1 \fSigma_1\vert\vert_2^{-2}\I{k} +\fSigma_1^{-2}) \\
\end{align*}
where we use at the first equality that 
$$ \lambda_\text{min}(\fM^{-1}) = \frac{1}{\lambda_\text{max}(\fM)} = \frac{1}{||\fM||_2} = ||\fM||_2^{-1}  
$$ for every invertible matrix $\fM \in \R^{n\times n}, n \in \N$ and at the last equality that the $j$-th eigenvalue of the diagonal matrix $\vert\vert\fH_1 \fSigma_1\vert\vert_2^{-2}\I{k} +\fSigma_1^{-2}$ is its $j$-th diagonal value which is  $\vert\vert\fH_1 \fSigma_1\vert\vert_2^{-2} + \lambda_j(\fSigma_1^{-2}).$
Therefore, the eigenvalues of the matrix $(\fSigma_1\rH{\fH_1}\fH_1 \fSigma_1)^{-1} + \fSigma_1^{-2})$ are larger than the eigenvalues of the matrix $(\vert\vert\fH_1 \fSigma_1\vert\vert^{-2}\I{k} +\fSigma_1^{-2})$. Therefore, the eigenvalues of $((\fSigma_1\rH{\fH_1}\fH_1 \fSigma_1)^{-1} + \fSigma_1^{-2})^{-1}$ are smaller than the eigenvalues of $(\vert\vert\fH_1 \fSigma_1\vert\vert^{-2}\I{k} +\fSigma_1^{-2})^{-1}$ and consequently also the trace.
Therefore
\begin{align*}
\tr((\fSigma_1\rH{\fH_1}\fH_1 \fSigma_1)^{-1} + \fSigma_1^{-2})^{-1}) &\leq  \tr((\vert\vert\fH_1 \fSigma_1\vert\vert_2^{-2}\I{k} + \fSigma_1^{-2})^{-1}) \\
&= \tr((\fSigma_1^{-1}(\vert\vert\fH_1 \fSigma_1\vert\vert_2^{-2}\fSigma_1^{2} + \I{k})\fSigma_1^{-1})^{-1})\\
&= \tr(\fSigma_1(\vert\vert\fH_1 \fSigma_1\vert\vert_2^{-2}\fSigma_1^{2} + \I{k})^{-1}\fSigma_1) \\
&= \vert\vert\fH_1 \fSigma_1\vert\vert_2^2\tr(\fSigma_1(\fSigma_1^{2} + \vert\vert\fH_1 \fSigma_1\vert\vert_2^2\I{k})^{-1}\fSigma_1)\\
&= \vert\vert\fH_1 \fSigma_1\vert\vert_2^2\sum\limits_{j = 1}^k \frac{\sigma_j^2}{\vert\vert\fH_1 \fSigma_1\vert\vert_2^2 + \sigma_j^2} \\
&\leq \vert\vert\fH_1 \fSigma_1\vert\vert_2^2 k \frac{\sigma_1^2}{\vert\vert\fH_1 \fSigma_1\vert\vert_2^2 + \sigma_1^2}\\
& =  \frac{k\vert\vert\fH_1 \fSigma_1\vert\vert_2^2}{\vert\vert\fH_1 \fSigma_1\vert\vert_2^2\sigma_1^{-2} +1}.
\end{align*}
Lastly, one has to estimate the norm $\vert\vert\fH_1\fSigma_1\vert\vert_2.$ It holds, that 
\begin{align*}
\vert\vert\fH_1\fSigma_1\vert\vert_2 & = \left\vert\left\vert\fSigma_3(\rT{\fSigma_3}\fSigma_3)^{\qpow}\fOmega_2\fOmega_1^\dagger\begin{pmatrix}
\fSigma_1^{-2\qpow+1} \\ \Z{l-s-k,k} \end{pmatrix}\fSigma_1\right\vert\right\vert_2 ß\\
&\leq \vert\vert\fSigma_3(\rT{\fSigma_3}\fSigma_3)^{\qpow}\vert\vert_2\vert\vert\fOmega_2\vert\vert_2\vert\vert\fOmega_1^\dagger\vert\vert_2 \vert\vert\fSigma_1^{-2q}\vert\vert_2 \\
&\leq \sigma_{l-s+1}\left(\frac{\sigma_{l-s+1}}{\sigma_{k}}\right)^{2\qpow}\vert\vert\fOmega_2\vert\vert_2\vert\vert\fOmega_1^\dagger\vert\vert_2,
\end{align*}
which concludes the proof.
\end{proof}
\begin{remark}
By a density argument this result also holds for general matrices $\fA\in \Cn^{m\times n}$ with $\mathrm{rank}(\fA)<\min(m,n)$ .
\end{remark}
Now, combining \Cref{lemma3,lemma4,lemma5} with \Cref{propSRFT} we prove the following: 
\begin{theorem}\label{thmadvbound}
If $4(\sqrt{k}+\sqrt{8\log(k\ns)})^2 \log(k) \leq l \leq \ns $, then the rcSVD basis matrix $\VrcSVD \in \R^{2N \times 2k}$ satisfies with failure probability $2/k$
\begin{align}
\vert\vert\Xs - \VrcSVD \rT{\VrcSVD}\Xs\vert\vert_F 
&\leq \sqrt{\left(\sum\limits_{j \geq k+1} \sigma_j^2 \right)+ \frac{\alpha^2\vert\vert\fOmega_2\vert\vert_2^2\vert\vert\fOmega_1^\dagger\vert\vert_2^2}{1+\gamma^2\vert\vert\fOmega_2\vert\vert_2^2\vert\vert\fOmega_1^\dagger\vert\vert_2^2}}  \\
&\leq\sqrt{ \left(\sum\limits_{j \geq k+1} \sigma_j^2\right) + \frac{6\sigma_{l+1}^2\left(\frac{\sigma_{l+1}}{\sigma_{k}}\right)^{4\qpow}k\ns/l}{1+6\frac{\sigma_{l+1}^2}{\sigma_1^2}\left(\frac{\sigma_{l+1}}{\sigma_{k}}\right)^{4\qpow}\ns/l}}\\
&\leq\sqrt{ \left(\sum\limits_{j \geq k+1} \sigma_j^2\right) + 6\sigma_{l+1}^2\left(\frac{\sigma_{l+1}}{\sigma_{k}}\right)^{4\qpow}k\ns/l}.
\end{align}
with $\sigma_j,j = 1,..,\ns$ the non-increasing sequence of singular values of the complex snapshot matrix 
$\Xc = \fQ + \im \fP,$ $\Xs\in \R^{2N \times \ns}.$
\end{theorem}
\begin{proof}
The first inequality follows directly from the \Cref{lemma3,lemma4,lemma5}, the second one by bounding $\vert\vert\fOmega_2\vert\vert_2^2\leq \ns/l$ and $\vert\vert\fOmega_1^\dagger\vert\vert_2\leq 6$. The third one follows because $1+6\frac{\sigma_{l+1}}{\sigma_1}\left(\frac{\sigma_{l+1}}{\sigma_{k}}\right)^{4\qpow}\ns/l \geq 1.$
\end{proof}

\begin{remark}
These results also hold (by setting $\fQ = \UY$) for the truncated projection $$\vert\vert(\I{m} - \Ucr\rH{(\Ucr)})\fA\vert\vert_F^2 = \vert\vert(\I{m} -  \UY{\fU_\fB}_{(k)}\rH{{\fU_\fB}_{(k)}}\rH{\UY})\fA\vert\vert_F = \vert\vert\fA - \fQ\fB_{(k)}\vert\vert_F.$$
\end{remark}
\begin{remark}
The parameter $s$ is not a methodical parameter but can be used to optimize the bound as the norms of the random matrices depend on $s$. For Gaussian matrices $s \geq 2$ leads to lower norms of $\vert\vert\fOmega_2\vert\vert_2,\vert\vert\fOmega_1^\dagger\vert\vert_2$. For the SRFT we only have a bound for s = 0 (\Cref{propSRFT} or \cite[Theorem 11.2]{Halko2011}). 
\end{remark}
To understand the influence of power iterations, we compare this bound to the bound from \Cref{quasi_opt} which states

\begin{align*}
\vert\vert\Xs - \VrcSVD \rT{\VrcSVD}\Xs\vert\vert_F &\leq (\sqrt{1 +6\ns/l} + 1) \sqrt{\sum\limits_{j \geq k+1} \sigma_j^2} \\
&= \sqrt{\left(\sum\limits_{j \geq k+1} \sigma_j^2 \right)+6\ns/l\sum\limits_{j \geq k+1} \sigma_j^2} + \sqrt{\sum\limits_{j \geq k+1} \sigma_j^2}
\end{align*}

whereas with the Theorem from \cite{Gu2015} we get

\begin{align}
\vert\vert\Xs - \VrcSVD \rT{\VrcSVD}\Xs\vert\vert_F 
&\leq\sqrt{ \left(\sum\limits_{j \geq k+1} \sigma_j^2\right) + (6\ns/l)k\sigma_{l+1}^2\left(\frac{\sigma_{l+1}}{\sigma_{k}}\right)^{4\qpow}}. \label{simplbound}
\end{align}
Here, we do not have the additional term $ \sqrt{\sum\limits_{j \geq k+1} \sigma_j^2}$ that resulted from the truncation step (\ref{eqn:truncbound}). Furthermore, the second term under the  square root in \Cref {simplbound} has  the squared $(l+1)$-th singular value instead of the sum of all $\sigma_j^2, j\geq k+1$. Moreover, we have the additional factor $k\left(\frac{\sigma_{l+1}}{\sigma_{k}}\right)^{4\qpow}$. Thus, if this factor is smaller than one, the bound is sharper than the bound from \Cref{quasi_opt}. We can further simplify the bound from \Cref{simplbound} to
\begin{align}
\vert\vert\Xs - \VrcSVD \rT{\VrcSVD}\Xs\vert\vert_F 
&\leq\sqrt{ \left(\sum\limits_{j \geq k+1} \sigma_j^2\right) + (6\ns/l)k\sigma_{l+1}^2\left(\frac{\sigma_{l+1}}{\sigma_{k}}\right)^{4\qpow}}\\ &\leq \sqrt{1 + 6\ns\left(\frac{\sigma_{l+1}}{\sigma_{k}}\right)^{4\qpow}}\sqrt{\sum\limits_{j \geq k+1} \sigma_j^2}. \label{factor}
\end{align}
where the last estimate makes use of $$\sigma_{l+1}^2 \leq \sum\limits_{j \geq k+1} \sigma_j^2,$$ 
as $\sigma_{l+1}^2$ is appearing in that sum, and we use $k/l \leq 1.$

Here, we see that the factor $\sqrt{1 + 6\ns\left(\frac{\sigma_{l+1}}{\sigma_{k}}\right)^{4\qpow}}$ in \Cref{factor} converges to 1 for $q \to\infty$ if we assume there is a gap between the $k$-th and $(l+1)$-th singular value.

\section{Formulation of rcSVD based on real numbers}\label{sec:interpretability}
For the cSVD algorithm we know that there is an equivalent algorithm that works only with real matrices \cite{Buchfink2019} (the \textit{cSVD via POD of $\fYs$}). In this section, we show that also  the rcSVD algorithm can be reformulated into a version that works only with real matrices: 
\begin{proposition}
Given the snapshot matrix $\Xs \in \R^{2N\times\ns}$, basis  size $2k$, oversampling parameter $\povs$ and  power iteration number $\qpow$ define the sketched extended snapshot matrix $\fZ = \fYs (\rT \fYs \fYs)^{\qpow} \widetilde \fOmega$ with $\fYs$ the extended snapshot matrix $\fYs = [\Xs, \JtN\Xs]$ and 
$\widetilde \fOmega$ the block-structured random matrix
$$\widetilde \fOmega :=
\left[\begin{pmatrix}
\mathrm{Re}(\fOmega)\\
\mathrm{Im}(\fOmega) 
\end{pmatrix}, 
\TJtN\begin{pmatrix}
\mathrm{Re}(\fOmega)\\
\mathrm{Im}(\fOmega) 
\end{pmatrix}\right]$$ with $\fOmega \in \Cn^{\ns \times (k  + \povs)}.$
We assume that $2k$ is such that there is a gap in the singular values of $\fZ$, i.e., $\sigma_{2k} (\fZ) > \sigma_{2k+1}(\fZ)$. Then, \texttt{rcSVD}($\Xs, 2k,\povs,  \qpow$) can be computed as \texttt{POD}($\fZ, 2k$)
\end{proposition}
\begin{proof}
The main step of the rcSVD procedure is to compute an SVD of the complex matrix $\fY := \Xc (\rH \Xc \Xc)^{\qpow}\fOmega$. According to \cite{Buchfink2019} this is equivalent to computing an SVD of $$\fZ:=[[\text{Re}(\fY); \text{Im}(\fY)], \TJtN [\text{Re}(\fY); \text{Im}(\fY)]]$$ if there is a gap in the singular values of $\fZ$.

With the definition $\fM:= \Xc (\rH \Xc \Xc)^{\qpow}$ we get
$$\text{Re}(\fY) = \text{Re}(\fM)\text{Re}(\fOmega) - \text{Im}(\fM)\text{Im}(\fOmega)$$ 
and 
$$\text{Im}(\fY) = \text{Re}(\fM)\text{Im}(\fOmega) + \text{Im}(\fM)\text{Re}(\fOmega)$$
we can reformulate
\begin{align*}
\begin{pmatrix}
\text{Re}(\fY) \\
\text{Im}(\fY)
\end{pmatrix}
&= \begin{pmatrix}
\text{Re}(\fM) & -\text{Im}(\fM) \\
\text{Im}(\fM) & \text{Re}(\fM)
\end{pmatrix}
\begin{pmatrix}
\text{Re}(\fOmega)\\
\text{Im}(\fOmega)
\end{pmatrix}
\end{align*}
and 
\begin{align*}
\TJtN
\begin{pmatrix}
\text{Re}(\fY) \\
\text{Im}(\fY)
\end{pmatrix}
=
\begin{pmatrix}
-\text{Im}(\fY) \\
\text{Re}(\fY)
\end{pmatrix}
= \begin{pmatrix}
\text{Re}(\fM) & -\text{Im}(\fM) \\
\text{Im}(\fM) & \text{Re}(\fM)
\end{pmatrix}
\begin{pmatrix}
-\text{Im}(\fOmega)\\
\text{Re}(\fOmega)
\end{pmatrix}.
\end{align*}
From this, it follows that
\begin{align}
\fZ&=[[\text{Re}(\fY); \text{Im}(\fY)], \TJtN [\text{Re}(\fY); \text{Im}(\fY)]] \\
&= 
\begin{pmatrix}
\text{Re}(\fM) & -\text{Im}(\fM) \\
\text{Im}(\fM) & \text{Re}(\fM)
\end{pmatrix}
\begin{pmatrix}
\text{Re}(\fOmega)&-\text{Im}(\fOmega)\\
\text{Im}(\fOmega)&\text{Re}(\fOmega)
\end{pmatrix} \\
& = 
\left[\begin{pmatrix}
\text{Re}(\fM)\\
\text{Im}(\fM) 
\end{pmatrix}, 
\TJtN\begin{pmatrix}
\text{Re}(\fM)\\
\text{Im}(\fM) 
\end{pmatrix}\right] \widetilde \fOmega,
\label{eqnrealint1}
\end{align}
where $\widetilde \fOmega$ is defined as $$\widetilde \fOmega :=
\left[\begin{pmatrix}
\text{Re}(\fOmega)\\
\text{Im}(\fOmega) 
\end{pmatrix}, 
\TJtN\begin{pmatrix}
\text{Re}(\fOmega)\\
\text{Im}(\fOmega) 
\end{pmatrix}\right].$$
Further it holds that
\begin{align*}
\left[\begin{pmatrix}
\text{Re}(\fM)\\
\text{Im}(\fM) 
\end{pmatrix}, 
\TJtN\begin{pmatrix}
\text{Re}(\fM)\\
\text{Im}(\fM) 
\end{pmatrix}\right]  = \fYs (\rT \fYs \fYs)^{\qpow}
\end{align*}
with $$\fYs:= \begin{pmatrix}
\fQ & -\fP\\
\fP &\fQ
\end{pmatrix}.$$
This can be seen by induction: 
Clearly this equation holds for $\qpow = 0$ by definition of $\Xc$ and $\fYs$. We now assume that the equation holds for some $\qpow \in \N_0.$
Then
\begin{align*}
\fYs (\rT \fYs \fYs)^{\qpow+1} &=  \fYs (\rT \fYs \fYs)^{\qpow}\rT \fYs \fYs \\
&= \left[\begin{pmatrix}
\text{Re}(\fM)\\
\text{Im}(\fM) 
\end{pmatrix}, 
\TJtN\begin{pmatrix}
\text{Re}(\fM)\\
\text{Im}(\fM) 
\end{pmatrix}\right]\rT \fYs \fYs \\
&= \begin{pmatrix}
\text{Re}(\fM) &-\text{Im}(\fM)\\
\text{Im}(\fM) &\text{Re}(\fM) 
\end{pmatrix}\rT \fYs \fYs.
\end{align*}
Furthermore, we have
\begin{align*}
\rT \fYs \fYs &=
\begin{pmatrix}
\rT\fQ & \rT\fP\\
-\rT\fP &\rT\fQ
\end{pmatrix}
\begin{pmatrix}
\fQ & -\fP\\
\fP &\fQ
\end{pmatrix}
=
\begin{pmatrix}
\rT\fQ\fQ + \rT\fP\fP & -\rT\fQ\fP + \rT\fP\fQ \\
\rT\fQ\fP - \rT\fP\fQ &
\rT\fQ\fQ + \rT\fP\fP
\end{pmatrix} \\
&=
\begin{pmatrix}
\text{Re}(\rH\Xc\Xc) &-\text{Im}(\rH\Xc\Xc) \\
\text{Im}(\rH\Xc\Xc) &\text{Re}(\rH\Xc\Xc)
\end{pmatrix}.
\end{align*}
Therefore, it holds that 
\begin{align*}
&\fYs (\rT \fYs \fYs)^{\qpow+1} = \begin{pmatrix}
\text{Re}(\fM) &-\text{Im}(\fM)\\
\text{Im}(\fM) &\text{Re}(\fM) 
\end{pmatrix}
\begin{pmatrix}
\text{Re}(\rH\Xc\Xc) &-\text{Im}(\rH\Xc\Xc) \\
\text{Im}(\rH\Xc\Xc) &\text{Re}(\rH\Xc\Xc)
\end{pmatrix}\\
&= \begin{pmatrix}
\fM_1 & \fM_2 \\
\fM_3 & \fM_4 
\end{pmatrix}\\
&= \begin{pmatrix}
\text{Re}(\fM\rH\Xc\Xc) 
&-\text{Im}(\fM\rH\Xc\Xc)  \\
\text{Im}(\fM\rH\Xc\Xc)  
& \text{Re}(\fM\rH\Xc\Xc) 
\end{pmatrix}\\
&=\left[\begin{pmatrix}
\text{Re}(\Xc(\rH \Xc \Xc)^{\qpow+1})\\
\text{Im}(\Xc(\rH \Xc \Xc)^{\qpow+1}) 
\end{pmatrix}, 
\TJtN\begin{pmatrix}
\text{Re}(\Xc(\rH \Xc \Xc)^{\qpow+1})\\
\text{Im}(\Xc(\rH \Xc \Xc)^{\qpow+1}) 
\end{pmatrix}\right],
\end{align*}
where \begin{align*}
\fM_1 &= \text{Re}(\fM)\text{Re}(\rH\Xc\Xc) -\text{Im}(\fM)\text{Im}(\rH\Xc\Xc) \\
\fM_2 &=-\text{Im}(\fM)\text{Re}(\rH\Xc\Xc) -\text{Re}(\fM)\text{Im}(\rH\Xc\Xc)  \\
\fM_3 &= \text{Im}(\fM)\text{Re}(\rH\Xc\Xc) +\text{Re}(\fM)\text{Im}(\rH\Xc\Xc)  \\
\fM_4 &= \text{Re}(\fM)\text{Re}(\rH\Xc\Xc) -\text{Im}(\fM)\text{Im}(\rH\Xc\Xc).
\end{align*}
Thus by the induction principle it follows that
\begin{align*}
\left[\begin{pmatrix}
\text{Re}(\fM)\\
\text{Im}(\fM) 
\end{pmatrix}, 
\TJtN\begin{pmatrix}
\text{Re}(\fM)\\
\text{Im}(\fM) 
\end{pmatrix}\right]  = \fYs (\rT \fYs \fYs)^{\qpow}
\end{align*}
for all $\qpow \in \N_0.$
Plugging this result into (\ref{eqnrealint1}) yields
\begin{align*}
\fZ & =  \fYs (\rT \fYs \fYs)^{\qpow} \widetilde \fOmega.
\end{align*}
\end{proof}

Therefore, the rcSVD can also be understood as \textit{rcSVD via POD of $\fYs (\rT \fYs \fYs)^{\qpow} \widetilde \fOmega$} or \textit{rcSVD via rPOD of $\fYs$} using a special block-structured random matrix $\widetilde \fOmega$ for the random sketching. This can be a useful equivalent characterization that allows results for real matrices to be applied when they are not available for complex matrices. Additionally, a numerical advantage may be a more easy  implementation as only real instead of complex arithmetics is required.

\section{Numerical Experiments}\label{Numerics}
To analyze what are practical choices for the oversampling parameter $\povs$ and the number of power iterations $\qpow$, we perform numerical experiments on a 2D wave equation problem. This example has originally been used in \cite{Peng2016} as a non-parametric one-dimensional problem and has been extended to the parametric two-dimensional case in \cite{Herkert2023dict}. The problem reads: Find the solution $u(t, \fxi)$ with spatial variable $$\fxi := (\xi_1, \xi_2) \in \varOmega := (0, 0.5) \times (0, 3)$$ and temporal variable $ t \in \It(\fmu) := [\tInit, \tEnd(\fmu)],$ with $\tInit = 0,\, \tEnd(\fmu) = 2/\fmu$ of

\begin{align*}
u_{tt}(t, \fxi) &= c^2 \Delta u(t, \fxi) &&&\textrm{in } \It(\fmu)\times\varOmega\\
u(t_0, \fxi) &= u^0(\fxi) := h(s(\fxi)), &&&\textrm{in } \varOmega,\\
u_t (t_0, \fxi) &= v^0(\fxi) &&&\textrm{in } \varOmega,\\
u(t, \fxi) &= 0 &&&\textrm{in }  \It(\fmu)\times\partial\varOmega, 
\end{align*}
where
\[s(\fxi): = 4\left\vert\left(\xi_2 + \frac{l}{2} - \frac{u^0_{sup}}{2}\right)\bigg/u^0_{sup}\right\vert, \ 
h (s): = \begin{cases}
  1 - \frac{3}{2}s^2 + \frac{3}{4}s^3,
& 0 \leq s \leq 1, \\
  \frac{1}{4} (2 - s)^3,
& 1 < s \leq 2,  \\
  0,
& s > 2,
\end{cases}  \]
and
\[v^0(\fxi):=\begin{cases}-\frac{4c}{u^0_{sup}} d_h(s(\fxi)), &\xi_2 + \frac{l}{2} - \frac{u^0_{sup}}{2} \geq 0\\
\frac{4c}{u^0_{sup}} d_h(s(\fxi)), &\xi_2 + \frac{l}{2} - \frac{u^0_{sup}}{2}<0
\end{cases},\] 
\[d_h (s): =
\begin{cases}
  (- 3s + \frac{9}{4}^2),
& 0 \leq s \leq 1, \\
  \frac{3}{4}(2 - s)^2,
& 1 < s\leq 2,  \\
  0,
&  s > 2.
\end{cases}  \]
We fix $u_{sup}^0 = 2$ and choose $\fmu = c \in\mathcal{P} := [1,2]$ as parameter (vector).
Central finite differences are used for the spatial discretization and the system is transformed into a first order ODE. This leads to the Hamiltonian system
\begin{equation}\label{wave_discr}
\ddt \fx(t; \fmu)  = \JtN \grad[\fx] \Ham(\fx(t; \fmu); \fmu) = \JtN\fH(\fmu)\fx(t; \fmu), \quad \fx(0; \fmu) = \fxInit(\fmu),
\end{equation}
where
$$\fH(\fmu) =\begin{pmatrix}
\fmu^2(\fD_{{\xi_1}{\xi_1}}+\fD_{{\xi_2}{\xi_2}}) \ &\Z{N} \\
 \Z{N}\  & \I{N}
\end{pmatrix}$$ 
and
$$\fxInit(\fmu) = [u^0(\fxi_1), ...,u^0(\fxi_N),v^0(\fxi_1), ...,v^0(\fxi_N)]$$
with $\{ \fxi_i \}_{i=1}^N \subset \varOmega$ being the grid points.
We denote the three-point central difference approximations in $\xi_1$-direction and in $\xi_2$-direction with the positive definite matrices $\fD_{{\xi_1}{\xi_1}}\in \R^{N \times N}$ and $\fD_{{\xi_2}{\xi_2}} \in \R^{N \times N}.$  Here, the generalized position and generalized momentum are the displacement at each grid point and the velocity at each grid point. 
The number of grid points including boundary points in $\xi_1$ is chosen as $N_{\xi_1} = 50$ and the number of grid points in $\xi_2$-direction is chosen as $N_{\xi_2} = 300$. The grid points are distributed equidistantly along each axis. This results in a Hamiltonian system of dimension of $2N = 15000$ with Hamiltonian 
$$\Ham(\fx, \fmu) = \frac{1}{2}\rT\fx\fH(\fmu)\fx. $$ \\
The implicit midpoint rule is a symplectic integrator \cite{hairer2006} that preserves quadratic Hamiltonians.
Moreover, we choose it with $n_t = 1500 $ equidistant time steps for temporal discretization. Since $\tEnd(\fmu)$ is parameter dependent this leads to different time step sizes for different parameters.  The parameters $\fmu_j =1 + 0.1 j,\ j = 0,...10 $ are used for the computation of the snapshot matrix $\Xs \in \R^{15000 \times 16500}$.  For the first experiment we compare the projection errors
$$e_\mathrm{proj}(\fV) = ||\Xs - \fV\rT{\fV}\Xs||_F^2$$
of the rcSVD bases $\fV \in \R^{2N \times k}$ for $k = 10, 20, 40, 80, 160$, $\qpow = 0,2,5,$ and $\povs = 5, 20,  l-k$ with   $l = 4(\sqrt{k}+\sqrt{8\log(k\ns)})^2 \log(k)$. We further include the projection error of the cSVD basis for comparison of the approximation quality and present the results in \Cref{fig:errors}.
\begin{figure}[h]
\begin{minipage}[c]{0.95\textwidth}
\begin{tikzpicture}

\definecolor{darkgray176}{RGB}{176,176,176}
\definecolor{darkorange25512714}{RGB}{255,127,14}
\definecolor{forestgreen4416044}{RGB}{44,160,44}
\definecolor{lightgray204}{RGB}{204,204,204}
\definecolor{steelblue31119180}{RGB}{31,119,180}

\begin{groupplot}[group style={group name=err_plots, group size=2 by 2}, scale = 0.7]
\nextgroupplot[
log basis x={10},
log basis y={10},
tick align=outside,
tick pos=left,
legend to name=leg0,
legend cell align={left},
title={\scriptsize Errors over basis size, \(\displaystyle q_{pow}\) = 0},
x grid style={darkgray176},
xlabel={basis size},
xmin=8.70550563296124, xmax=183.791736799526,
xmode=log,
xtick style={color=black},
xtick={0.1,1,10,100,1000,10000},
xticklabels={
  \(\displaystyle {10^{-1}}\),
  \(\displaystyle {10^{0}}\),
  \(\displaystyle {10^{1}}\),
  \(\displaystyle {10^{2}}\),
  \(\displaystyle {10^{3}}\),
  \(\displaystyle {10^{4}}\)
},
y grid style={darkgray176},
ylabel={projection error},
ymin=5.45507639350429e-10, ymax=827.07216889748,
ymode=log,
ytick style={color=black},
ytick={1e-13,1e-10,1e-07,0.0001,0.1,100,100000,100000000},
yticklabels={
  \(\displaystyle {10^{-13}}\),
  \(\displaystyle {10^{-10}}\),
  \(\displaystyle {10^{-7}}\),
  \(\displaystyle {10^{-4}}\),
  \(\displaystyle {10^{-1}}\),
  \(\displaystyle {10^{2}}\),
  \(\displaystyle {10^{5}}\),
  \(\displaystyle {10^{8}}\)
}
]
\addlegendentry{cSVD}
\addlegendentry{rcSVD, $p_\text{pow} = 5$}
\addlegendentry{rcSVD, $p_\text{pow} = 20$}
\addlegendentry{rcSVD, $p_\text{pow} = l-k$}
\addplot [red, mark=square*, mark size=4, mark options={solid}]
table {%
10 275.690722965827
20 40.8957757119021
40 6.74195630292087
80 1.02819073524068
160 1.63652291805129e-09
};
\addplot [steelblue31119180, dash pattern=on 11.1pt off 4.8pt, mark=triangle*, mark size=4, mark options={solid,rotate=180}]
table {%
10 322.973999265143
20 65.9673669840127
40 13.5213794095001
80 2.28562001068129
160 5.48385025749044e-09
};
\addplot [darkorange25512714, dash pattern=on 19.2pt off 4.8pt on 3pt off 4.8pt, mark=triangle*, mark size=4, mark options={solid}]
table {%
10 276.538954863481
20 42.8645827392132
40 8.37642744940887
80 1.52143846428038
160 3.97081987983135e-09
};
\addplot [forestgreen4416044, dash pattern=on 3pt off 4.95pt, mark=triangle*, mark size=4, mark options={solid,rotate=270}]
table {%
10 275.690722965857
20 40.8957757118939
40 6.74195630292069
80 1.02819073524076
160 1.62392172448573e-09
};

\nextgroupplot[
log basis x={10},
log basis y={10},
tick align=outside,
tick pos=left,
title={\scriptsize Errors over basis size, \(\displaystyle q_{pow}\) = 2},
x grid style={darkgray176},
xlabel={basis size},
xmin=8.70550563296124, xmax=183.791736799526,
xmode=log,
xtick style={color=black},
xtick={0.1,1,10,100,1000,10000},
xticklabels={
  \(\displaystyle {10^{-1}}\),
  \(\displaystyle {10^{0}}\),
  \(\displaystyle {10^{1}}\),
  \(\displaystyle {10^{2}}\),
  \(\displaystyle {10^{3}}\),
  \(\displaystyle {10^{4}}\)
},
y grid style={darkgray176},
ylabel={projection error},
ymin=5.45507639350429e-10, ymax=827.07216889748,
ymode=log,
ytick style={color=black},
ytick={1e-13,1e-10,1e-07,0.0001,0.1,100,100000,100000000},
yticklabels={
  \(\displaystyle {10^{-13}}\),
  \(\displaystyle {10^{-10}}\),
  \(\displaystyle {10^{-7}}\),
  \(\displaystyle {10^{-4}}\),
  \(\displaystyle {10^{-1}}\),
  \(\displaystyle {10^{2}}\),
  \(\displaystyle {10^{5}}\),
  \(\displaystyle {10^{8}}\)
}
]
\addplot [red, mark=square*, mark size=4, mark options={solid}]
table {%
10 275.690722965827
20 40.8957757119021
40 6.74195630292087
80 1.02819073524068
160 1.63652291805129e-09
};
\addplot [steelblue31119180, dash pattern=on 11.1pt off 4.8pt, mark=triangle*, mark size=4, mark options={solid,rotate=180}]
table {%
10 275.690739468509
20 40.9025419916391
40 6.75586141471457
80 1.03617555430354
160 1.66130129199939e-09
};
\addplot [darkorange25512714, dash pattern=on 19.2pt off 4.8pt on 3pt off 4.8pt, mark=triangle*, mark size=4, mark options={solid}]
table {%
10 275.690722965824
20 40.8957757134044
40 6.74196214090027
80 1.02826665655755
160 1.65187728881599e-09
};
\addplot [forestgreen4416044, dash pattern=on 3pt off 4.95pt, mark=triangle*, mark size=4, mark options={solid,rotate=270}]
table {%
10 275.690722965857
20 40.8957757118937
40 6.74195630292063
80 1.02819073524075
160 1.67040382745823e-09
};
\nextgroupplot[
log basis x={10},
log basis y={10},
tick align=outside,
tick pos=left,
title={\scriptsize Errors over basis size, \(\displaystyle q_{pow}\) = 5},
x grid style={darkgray176},
xlabel={basis size},
xmin=8.70550563296124, xmax=183.791736799526,
xmode=log,
xtick style={color=black},
xtick={0.1,1,10,100,1000,10000},
xticklabels={
  \(\displaystyle {10^{-1}}\),
  \(\displaystyle {10^{0}}\),
  \(\displaystyle {10^{1}}\),
  \(\displaystyle {10^{2}}\),
  \(\displaystyle {10^{3}}\),
  \(\displaystyle {10^{4}}\)
},
y grid style={darkgray176},
ylabel={projection error},
ymin=5.45507639350429e-10, ymax=827.07216889748,
ymode=log,
ytick style={color=black},
ytick={1e-13,1e-10,1e-07,0.0001,0.1,100,100000,100000000},
yticklabels={
  \(\displaystyle {10^{-13}}\),
  \(\displaystyle {10^{-10}}\),
  \(\displaystyle {10^{-7}}\),
  \(\displaystyle {10^{-4}}\),
  \(\displaystyle {10^{-1}}\),
  \(\displaystyle {10^{2}}\),
  \(\displaystyle {10^{5}}\),
  \(\displaystyle {10^{8}}\)
}
]
\addplot [red, mark=square*, mark size=4, mark options={solid}]
table {%
10 275.690722965827
20 40.8957757119021
40 6.74195630292087
80 1.02819073524068
160 1.63652291805129e-09
};
\addplot [steelblue31119180, dash pattern=on 11.1pt off 4.8pt, mark=triangle*, mark size=4, mark options={solid,rotate=180}]
table {%
10 275.690722965814
20 40.8957759124426
40 6.74218883565127
80 1.02842080764319
160 1.64011020884186e-09
};
\addplot [darkorange25512714, dash pattern=on 19.2pt off 4.8pt on 3pt off 4.8pt, mark=triangle*, mark size=4, mark options={solid}]
table {%
10 275.690722965857
20 40.8957757118938
40 6.74195630292306
80 1.02819074156573
160 1.62904855754085e-09
};
\addplot [forestgreen4416044, dash pattern=on 3pt off 4.95pt, mark=triangle*, mark size=4, mark options={solid,rotate=270}]
table {%
10 275.690722965857
20 40.8957757118938
40 6.74195630292069
80 1.02819073524072
160 1.68540774595698e-09
};
\nextgroupplot[
hide x axis,
hide y axis,
legend cell align={left},
legend style={fill opacity=0.8, draw opacity=1, text opacity=1, draw=lightgray204},
tick align=outside,
tick pos=left,
x grid style={darkgray176},
xmin=0, xmax=1,
xtick style={color=black},
y grid style={darkgray176},
ymin=0, ymax=1,
ytick style={color=black}
]
\end{groupplot}

\node[
   right = 9.3cm,
   below = 6.0cm,
    align=left,
    draw=black,
    anchor=south,
  ] {%
  \ref*{leg0}
  };
\end{tikzpicture}
\end{minipage}

\caption{Projection error for different values for $\povs$ and $\qpow$}
\label{fig:errors}
\end{figure}
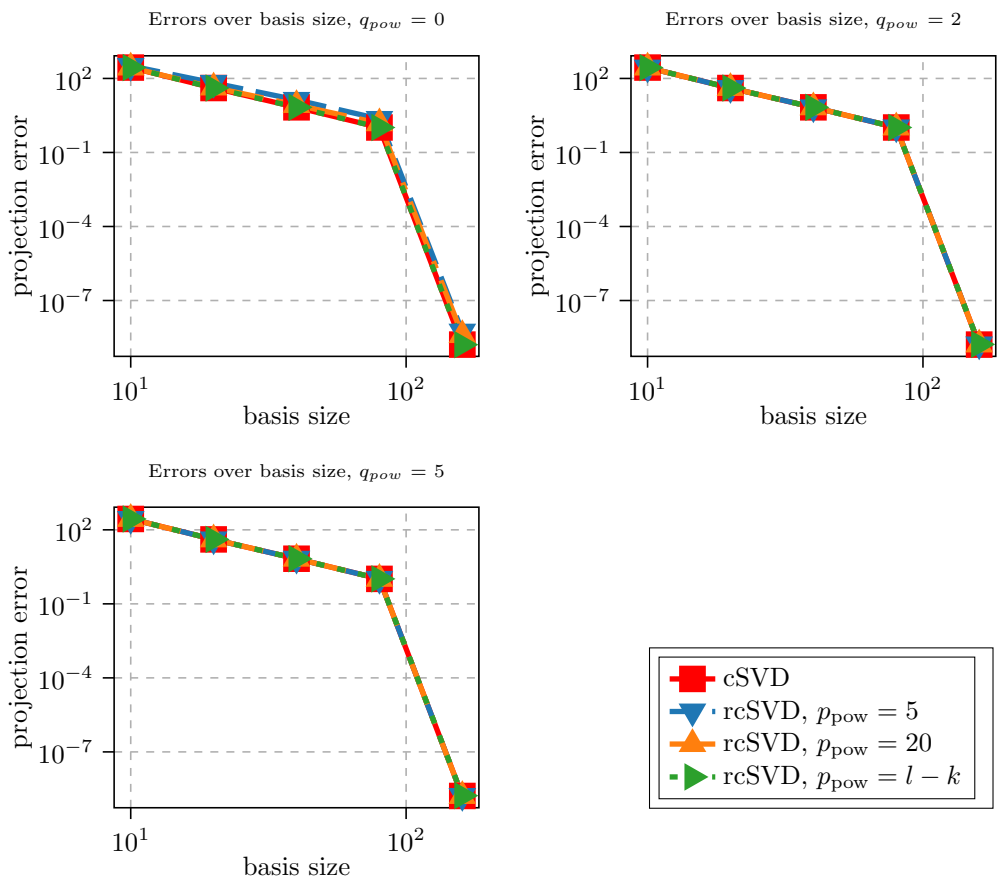

We observe that the rcSVD yields a very good approximation for all tested values of $\qpow$ and $\povs$. Especially when choosing $\qpow > 0$ the projection errors are almost equal to the projection error of cSVD. For $\qpow = 0$ we observe that increasing $\povs$ slightly improves the projection error $e_\mathrm{proj}(\fV)$. For higher values of $\qpow$ no influence of $\povs$ on the error can be observed since  it equals the best approximation error of cSVD already for $\povs = 5.$ Therefore, we conclude that in practice much smaller values for $\povs$ than $\povs = l-k$ with $l = 4(\sqrt{k}+\sqrt{8\log(k\ns)})^2 \log(k)$ can be used. 

In order to highlight the computational advantages of the randomized algorithms, we present the average runtimes (averaged over 5 runs each) in \Cref{runtime}. They are measured on a computer with 64 GB RAM and a 13th Gen Intel i7-13700K processor. The experiments are implemented in Python 3.8.10 using numpy 1.24.3 and scipy 1.10.1.

\begin{figure}[h]
\begin{tikzpicture}

\definecolor{darkgray176}{RGB}{176,176,176}
\definecolor{darkorange25512714}{RGB}{255,127,14}
\definecolor{forestgreen4416044}{RGB}{44,160,44}
\definecolor{lightgray204}{RGB}{204,204,204}
\definecolor{steelblue31119180}{RGB}{31,119,180}

\begin{groupplot}[group style={group size=2 by 2}, scale = 0.7]
\nextgroupplot[
log basis x={10},
log basis y={10},
tick align=outside,
tick pos=left,
title={\scriptsize Basis generation time, \(\displaystyle q_{pow}\) = 0},
x grid style={darkgray176},
xlabel={basis size},
xmin=8.70550563296124, xmax=183.791736799526,
xmode=log,
xtick style={color=black},
xtick={0.1,1,10,100,1000,10000},
xticklabels={
  \(\displaystyle {10^{-1}}\),
  \(\displaystyle {10^{0}}\),
  \(\displaystyle {10^{1}}\),
  \(\displaystyle {10^{2}}\),
  \(\displaystyle {10^{3}}\),
  \(\displaystyle {10^{4}}\)
},
y grid style={darkgray176},
ylabel={basis generation time},
ymin=2.24370986620585, ymax=24732.726274538,
ymode=log,
ytick style={color=black},
ytick={0.1,1,10,100,1000,10000,100000,1000000},
yticklabels={
  \(\displaystyle {10^{-1}}\),
  \(\displaystyle {10^{0}}\),
  \(\displaystyle {10^{1}}\),
  \(\displaystyle {10^{2}}\),
  \(\displaystyle {10^{3}}\),
  \(\displaystyle {10^{4}}\),
  \(\displaystyle {10^{5}}\),
  \(\displaystyle {10^{6}}\)
}
]
\addplot [red, mark=square*, mark size=4, mark options={solid}]
table {%
10 1241.60421490669
20 1241.60421490669
40 1241.60421490669
80 1241.60421490669
160 1241.60421490669
};
\addplot [steelblue31119180, mark=triangle*, mark size=4, mark options={solid,rotate=180}]
table {%
10 6.73112959861755
20 6.96477661132812
40 7.0986035823822
80 7.39678516387939
160 8.23601088523865
};
\addplot [darkorange25512714, mark=triangle*, mark size=4, mark options={solid}]
table {%
10 7.00982303619385
20 7.22446451187134
40 7.53470883369446
80 7.54338240623474
160 8.3738630771637
};
\addplot [forestgreen4416044, mark=triangle*, mark size=4, mark options={solid,rotate=270}]
table {%
10 30.4527867794037
20 56.456690788269
40 126.828850269318
80 340.287944078445
160 1314.84462714195
};

\nextgroupplot[
log basis x={10},
log basis y={10},
tick align=outside,
tick pos=left,
title={\scriptsize Basis generation time, \(\displaystyle q_{pow}\) = 2},
x grid style={darkgray176},
xlabel={basis size},
xmin=8.70550563296124, xmax=183.791736799526,
xmode=log,
xtick style={color=black},
xtick={0.1,1,10,100,1000,10000},
xticklabels={
  \(\displaystyle {10^{-1}}\),
  \(\displaystyle {10^{0}}\),
  \(\displaystyle {10^{1}}\),
  \(\displaystyle {10^{2}}\),
  \(\displaystyle {10^{3}}\),
  \(\displaystyle {10^{4}}\)
},
y grid style={darkgray176},
ylabel={basis generation time},
ymin=2.24370986620585, ymax=24732.726274538,
ymode=log,
ytick style={color=black},
ytick={0.1,1,10,100,1000,10000,100000,1000000},
yticklabels={
  \(\displaystyle {10^{-1}}\),
  \(\displaystyle {10^{0}}\),
  \(\displaystyle {10^{1}}\),
  \(\displaystyle {10^{2}}\),
  \(\displaystyle {10^{3}}\),
  \(\displaystyle {10^{4}}\),
  \(\displaystyle {10^{5}}\),
  \(\displaystyle {10^{6}}\)
}
]
\addplot [red, mark=square*, mark size=4, mark options={solid}]
table {%
10 1241.60421490669
20 1241.60421490669
40 1241.60421490669
80 1241.60421490669
160 1241.60421490669
};
\addplot [steelblue31119180, mark=triangle*, mark size=4, mark options={solid,rotate=180}]
table {%
10 9.29854021072388
20 9.74632134437561
40 10.3022086620331
80 12.2409708499908
160 15.5554538726807
};
\addplot [darkorange25512714, mark=triangle*, mark size=4, mark options={solid}]
table {%
10 10.1709331989288
20 10.3389545917511
40 11.3073535919189
80 12.7460961818695
160 16.1618017673492
};
\addplot [forestgreen4416044, mark=triangle*, mark size=4, mark options={solid,rotate=270}]
table {%
10 98.9421360015869
20 193.478971242905
40 425.863892221451
80 1105.33603916168
160 4088.02628817558
};

\nextgroupplot[
log basis x={10},
log basis y={10},
tick align=outside,
tick pos=left,
title={\scriptsize Basis generation time, \(\displaystyle q_{pow}\) = 5},
x grid style={darkgray176},
xlabel={basis size},
xmin=8.70550563296124, xmax=183.791736799526,
xmode=log,
xtick style={color=black},
xtick={0.1,1,10,100,1000,10000},
xticklabels={
  \(\displaystyle {10^{-1}}\),
  \(\displaystyle {10^{0}}\),
  \(\displaystyle {10^{1}}\),
  \(\displaystyle {10^{2}}\),
  \(\displaystyle {10^{3}}\),
  \(\displaystyle {10^{4}}\)
},
y grid style={darkgray176},
ylabel={basis generation time},
ymin=2.24370986620585, ymax=24732.726274538,
ymode=log,
ytick style={color=black},
ytick={0.1,1,10,100,1000,10000,100000,1000000},
yticklabels={
  \(\displaystyle {10^{-1}}\),
  \(\displaystyle {10^{0}}\),
  \(\displaystyle {10^{1}}\),
  \(\displaystyle {10^{2}}\),
  \(\displaystyle {10^{3}}\),
  \(\displaystyle {10^{4}}\),
  \(\displaystyle {10^{5}}\),
  \(\displaystyle {10^{6}}\)
}
]
\addplot [red, mark=square*, mark size=4, mark options={solid}]
table {%
10 1241.60421490669
20 1241.60421490669
40 1241.60421490669
80 1241.60421490669
160 1241.60421490669
};
\addplot [steelblue31119180, mark=triangle*, mark size=4, mark options={solid,rotate=180}]
table {%
10 13.1302813053131
20 14.3744291305542
40 15.1281248569489
80 19.4232367992401
160 26.4891467094421
};
\addplot [darkorange25512714, mark=triangle*, mark size=4, mark options={solid}]
table {%
10 14.8316363811493
20 15.3385501861572
40 17.5917099475861
80 20.602678155899
160 27.5311779022217
};
\addplot [forestgreen4416044, mark=triangle*, mark size=4, mark options={solid,rotate=270}]
table {%
10 205.853742694855
20 398.543378925324
40 875.758471202851
80 2256.05507693291
160 8244.24209151268
};

\nextgroupplot[
hide x axis,
hide y axis,
legend cell align={left},
legend style={fill opacity=0.8, draw opacity=1, text opacity=1, draw=lightgray204},
tick align=outside,
tick pos=left,
x grid style={darkgray176},
xmin=0, xmax=1,
xtick style={color=black},
y grid style={darkgray176},
ymin=0, ymax=1,
ytick style={color=black}
]

\end{groupplot}
\node[
   right = 9.3cm,
   below = 6.0cm,
    align=left,
    draw=black,
    anchor=south,
  ] {%
  \ref*{leg0}
  };
\end{tikzpicture}
\caption{Runtimes for different values for $\povs$ and $
\qpow$}
\label{runtime}
\end{figure}
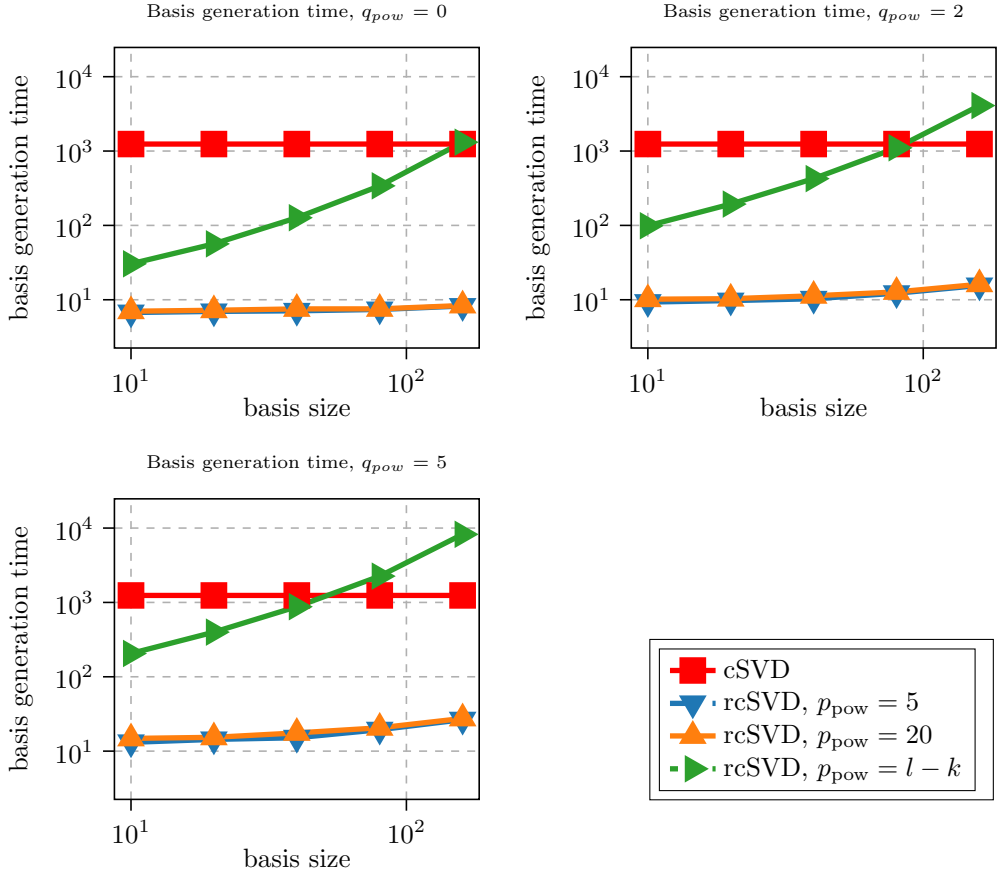
We observe that the rcSVD is highly efficient if $\povs$ is chosen as a small value that is independent from $k$. Choosing $\povs$ as suggested by \Cref{thmbound,thmadvbound} to obtain theoretical guarantees also yields an advantage regarding the use of computational resources but only for small basis sizes and small values of $\qpow$ as this $k$-dependent choice of $\povs$ drastically increases the runtime especially for larger values of $k$ compared to the runtimes for $\povs = 5,20$. In practice these small values for $\povs$ also result in very good approximations as we saw in \Cref{fig:errors} compared to cSVD while requiring less than $5\%$ of the computational costs.

For the next experiment we compute the effectivities

\begin{align*}
\eff_\mathrm{det} = \frac{e_\mathrm{proj}(\fV)}{\eta_\mathrm{det}(\Xs, \fOmega, k)}, \qquad
\eff_\mathrm{det}^{\ \mathrm{adv}} = \frac{e_\mathrm{proj}(\fV)}{\eta_\mathrm{det}^{\mathrm{adv}}(\Xs, \fOmega, k, \qpow)}
\end{align*}
 for the deterministic error bounds 
\begin{align*}\eta_\mathrm{det}(\Xs, \fOmega, k) &= \left(1 + \sqrt{1 + 
\vert\vert\fOmega_2\vert\vert_2^2\vert\vert\fOmega_1^\dagger\vert\vert_2^2} \right)\sqrt{\sum\limits_{j \geq k+1} \sigma_j^2} \\
\eta_\mathrm{det}^{\mathrm{adv}}(\Xs, \fOmega, k, \qpow) &= 
\sqrt{\left(\sum\limits_{j \geq k+1} \sigma_j^2 \right)+ \alpha^2\vert\vert\fOmega_2\vert\vert_2^2\vert\vert\fOmega_1^\dagger\vert\vert_2^2}\end{align*}
where with the superscript $adv$, we denote the advanced error bound that also takes the number of power iterations into account. The effectivity measures the extent of overestimation of the error bounds. Ideally, the effectivity is close to or even equal to one. An effectivity larger than one corresponds to overestimation and an effectivity lower than one means, that the error is underestimated, i.e., it is not bounded.  
Note, that these error bounds are expensive to evaluate since the singular values and singular vectors of the snapshot matrix are required. For efficient error estimation for example in combination with adaptive basis generation, error estimation techniques like in \cite{Rettberg2023} or \cite{Smetana19} have to be applied. 
We present average results over 5 runs (i.e., 5 draws of $\fOmega$) in \Cref{detbound}.

\begin{figure}[h]
\begin{tikzpicture}

\definecolor{darkgray176}{RGB}{176,176,176}
\definecolor{darkorange25512714}{RGB}{255,127,14}
\definecolor{forestgreen4416044}{RGB}{44,160,44}
\definecolor{lightgray204}{RGB}{204,204,204}
\definecolor{steelblue31119180}{RGB}{31,119,180}

\begin{groupplot}[group style={group size=2 by 2}, scale = 0.7]
\nextgroupplot[
legend to name=leg1,
legend cell align={left},
log basis x={10},
log basis y={10},
tick align=outside,
tick pos=left,
title={\scriptsize Effectivity det. error bounds, \(\displaystyle q_{pow}\) = 0},
x grid style={darkgray176},
xlabel={basis size},
xmajorgrids,
xmin=8.70550563296124, xmax=183.791736799526,
xmode=log,
xtick style={color=black},
xtick={0.1,1,10,100,1000,10000},
xticklabels={
  \(\displaystyle {10^{-1}}\),
  \(\displaystyle {10^{0}}\),
  \(\displaystyle {10^{1}}\),
  \(\displaystyle {10^{2}}\),
  \(\displaystyle {10^{3}}\),
  \(\displaystyle {10^{4}}\)
},
y grid style={darkgray176},
ylabel={effectivity},
ymajorgrids,
ymin=0.5, ymax=1335.82822746208,
ymode=log,
ytick style={color=black},
ytick={0.01,0.1,1,10,100,1000,10000,100000},
yticklabels={
  \(\displaystyle {10^{-2}}\),
  \(\displaystyle {10^{-1}}\),
  \(\displaystyle {10^{0}}\),
  \(\displaystyle {10^{1}}\),
  \(\displaystyle {10^{2}}\),
  \(\displaystyle {10^{3}}\),
  \(\displaystyle {10^{4}}\),
  \(\displaystyle {10^{5}}\)
}
]
\addlegendentry{$\eff_{\textrm{det}}^{\ \mathrm{adv}}$, $p_\text{pow} = 5$}
\addlegendentry{$\eff_{\textrm{det}}^{\ \mathrm{adv}}$, $p_\text{pow} = 20$}
\addlegendentry{$\eff_{\textrm{det}}^{\ \mathrm{adv}}$, $p_\text{pow} = l-k$}
\addlegendentry{$\eff_{\textrm{det}}$, $p_\text{pow} = 5$}
\addlegendentry{$\eff_{\textrm{det}}$, $p_\text{pow} = 20$}
\addlegendentry{$\eff_{\textrm{det}}$, $p_\text{pow} = l-k$}
\addplot [steelblue31119180, dash pattern=on 3pt off 4.95pt]
table {%
10 44.3169646893678
20 102.665766427157
40 168.25527853012
80 274.086608484773
160 109.928795218276
};
\addplot [darkorange25512714, dash pattern=on 3pt off 4.95pt]
table {%
10 2.96143652326528
20 13.6676745634098
40 37.1615969121125
80 89.1589559809971
160 41.9253006954611
};
\addplot [forestgreen4416044, dash pattern=on 3pt off 4.95pt]
table {%
10 0.99999999999989
20 1.0000000000002
40 1.00000000000003
80 0.999999999999919
160 1.00775972965664
};
\addplot [steelblue31119180]
table {%
10 93.6753462207432
20 100.774679278869
40 117.553300790981
80 141.000967318977
160 134.774472406722
};
\addplot [darkorange25512714]
table {%
10 48.123682184905
20 60.3727701742268
40 63.269707826656
80 76.411350528351
160 55.5709998812047
};
\addplot [forestgreen4416044]
table {%
10 4.61908446115475
20 3.92716721630678
40 3.39985195407518
80 2.98051151658389
160 2.66331252046772
};

\nextgroupplot[
log basis x={10},
log basis y={10},
tick align=outside,
tick pos=left,
title={\scriptsize Effectivity det. error bounds, \(\displaystyle q_{pow}\) = 2},
x grid style={darkgray176},
xlabel={basis size},
xmajorgrids,
xmin=8.70550563296124, xmax=183.791736799526,
xmode=log,
xtick style={color=black},
xtick={0.1,1,10,100,1000,10000},
xticklabels={
  \(\displaystyle {10^{-1}}\),
  \(\displaystyle {10^{0}}\),
  \(\displaystyle {10^{1}}\),
  \(\displaystyle {10^{2}}\),
  \(\displaystyle {10^{3}}\),
  \(\displaystyle {10^{4}}\)
},
y grid style={darkgray176},
ylabel={effectivity},
ymajorgrids,
ymin=0.5, ymax=1335.82822746208,
ymode=log,
ytick style={color=black},
ytick={0.01,0.1,1,10,100,1000,10000,100000},
yticklabels={
  \(\displaystyle {10^{-2}}\),
  \(\displaystyle {10^{-1}}\),
  \(\displaystyle {10^{0}}\),
  \(\displaystyle {10^{1}}\),
  \(\displaystyle {10^{2}}\),
  \(\displaystyle {10^{3}}\),
  \(\displaystyle {10^{4}}\),
  \(\displaystyle {10^{5}}\)
}
]
\addplot [steelblue31119180, dash pattern=on 3pt off 4.95pt]
table {%
10 1.00620564623951
20 8.562926239298
40 73.9700425410354
80 261.852517848587
160 236.035734659589
};
\addplot [darkorange25512714, dash pattern=on 3pt off 4.95pt]
table {%
10 1.000000000001
20 1.00000170644684
40 1.04347741958662
80 7.65464674985026
160 49.0024303177295
};
\addplot [forestgreen4416044, dash pattern=on 3pt off 4.95pt]
table {%
10 0.999999999999889
20 1.00000000000021
40 1.00000000000004
80 0.999999999999931
160 0.979716935000981
};
\addplot [steelblue31119180]
table {%
10 109.741448914051
20 162.528780046096
40 235.274627950772
80 311.023195916147
160 444.881990262736
};
\addplot [darkorange25512714]
table {%
10 48.2717468053711
20 63.2792399015535
40 78.6082903284461
80 113.059357764886
160 133.582822746208
};
\addplot [forestgreen4416044]
table {%
10 4.61908446115475
20 3.9271672163068
40 3.39985195407521
80 2.98051151658392
160 2.58920088064186
};

\nextgroupplot[
log basis x={10},
log basis y={10},
tick align=outside,
tick pos=left,
title={\scriptsize Effectivity det. error bounds, \(\displaystyle q_{pow}\) = 5},
x grid style={darkgray176},
xlabel={basis size},
xmajorgrids,
xmin=8.70550563296124, xmax=183.791736799526,
xmode=log,
xtick style={color=black},
xtick={0.1,1,10,100,1000,10000},
xticklabels={
  \(\displaystyle {10^{-1}}\),
  \(\displaystyle {10^{0}}\),
  \(\displaystyle {10^{1}}\),
  \(\displaystyle {10^{2}}\),
  \(\displaystyle {10^{3}}\),
  \(\displaystyle {10^{4}}\)
},
y grid style={darkgray176},
ylabel={effectivity},
ymajorgrids,
ymin=0.5, ymax=1335.82822746208,
ymode=log,
ytick style={color=black},
ytick={0.01,0.1,1,10,100,1000,10000,100000},
yticklabels={
  \(\displaystyle {10^{-2}}\),
  \(\displaystyle {10^{-1}}\),
  \(\displaystyle {10^{0}}\),
  \(\displaystyle {10^{1}}\),
  \(\displaystyle {10^{2}}\),
  \(\displaystyle {10^{3}}\),
  \(\displaystyle {10^{4}}\),
  \(\displaystyle {10^{5}}\)
}
]
\addplot [steelblue31119180, dash pattern=on 3pt off 4.95pt]
table {%
10 1.00000000006012
20 1.00512031394843
40 7.69673344284678
80 75.2054473875441
160 125.430762323125
};
\addplot [darkorange25512714, dash pattern=on 3pt off 4.95pt]
table {%
10 0.999999999999888
20 1.0000000000002
40 1.00000001187536
80 1.00551216189297
160 16.8696210029024
};
\addplot [forestgreen4416044, dash pattern=on 3pt off 4.95pt]
table {%
10 0.99999999999989
20 1.0000000000002
40 1.00000000000003
80 0.999999999999953
160 0.970995251432207
};
\addplot [steelblue31119180]
table {%
10 109.741455483114
20 162.555669928314
40 235.751744660294
80 313.368448046304
160 450.63009865211
};
\addplot [darkorange25512714]
table {%
10 48.2717468053652
20 63.279239903891
40 78.6083583967317
80 113.067705340753
160 135.454790496544
};
\addplot [forestgreen4416044]
table {%
10 4.61908446115475
20 3.92716721630679
40 3.39985195407518
80 2.98051151658399
160 2.56615117110823
};

\nextgroupplot[
hide x axis,
hide y axis,
legend cell align={left},
legend style={
  fill opacity=0.8,
  draw opacity=1,
  text opacity=1,
  at={(0.97,0.03)},
  anchor=south east,
  draw=lightgray204
},
tick align=outside,
tick pos=left,
x grid style={darkgray176},
xmin=0, xmax=1,
xtick style={color=black},
y grid style={darkgray176},
ymin=0, ymax=1,
ytick style={color=black}
]
\end{groupplot}
\node[
   right = 9.4cm,
   below = 6.0cm,
    align=left,
    draw=black,
    anchor=south,
  ] {%
  \ref*{leg1}
  };
\end{tikzpicture}
\caption{Effectivity of deterministic error bound for different values of $\povs$ and $\qpow$}
\label{detbound}
\end{figure}
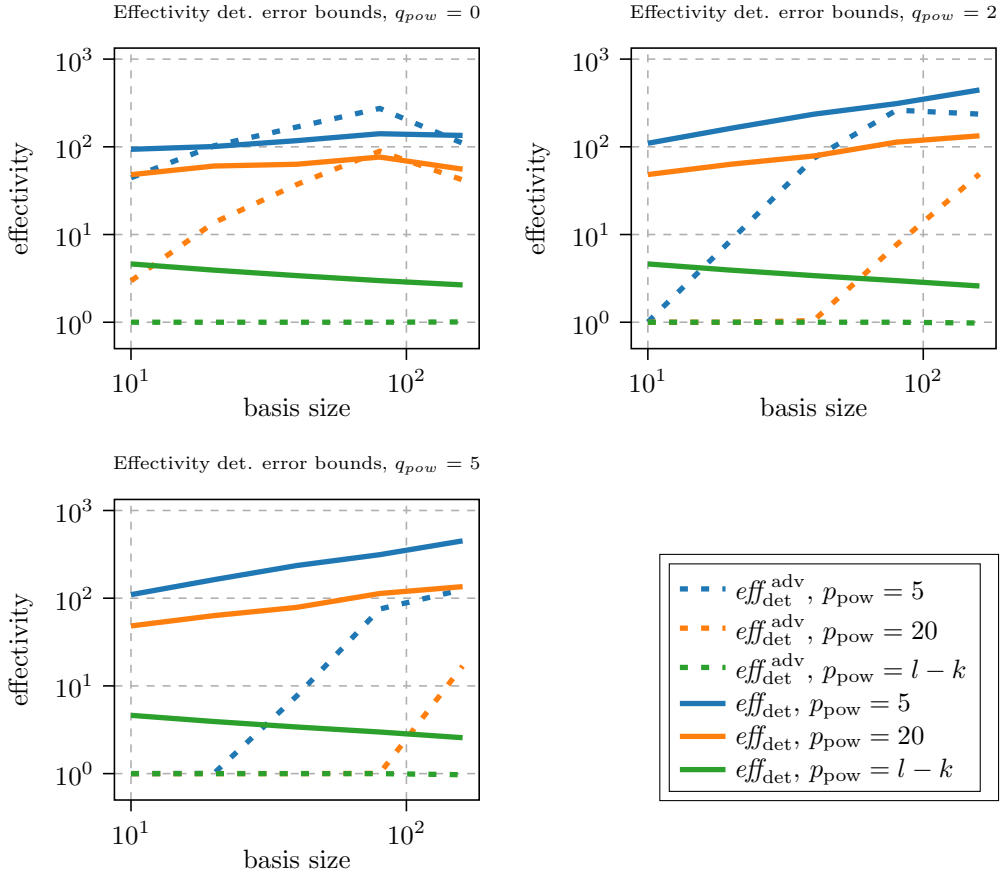
We observe that $\eff_\mathrm{det}^{\ \mathrm{adv}}$ is 
very close to one 
for small values of $k$ with an increasing factor of overestimation the higher $k$ is chosen. Further, $\eta_\mathrm{det}^{\mathrm{adv}}$ gets sharper the higher $\qpow$ and $\povs$ are chosen. Also $\eff_\mathrm{det}$ becomes closer to one the higher $\povs$ is chosen. However the value of $\qpow$ does not influence $\eff_\mathrm{det}$. For $\qpow = 0$ and $\povs = 5$ both bounds roughly have the same extent of overestimation i.e., the blue curves are close together. For increasing values of $\qpow$ and $\povs$ the advances bound gets sharper, i.e., the dotted lines are closer to one than the solid lines. 

For the next experiment we compute the effectivities

\begin{align*}
\eff_{\textrm{prob}} = \frac{e_\mathrm{proj}(\fV)}{\eta_{\textrm{prob}}(\Xs,k, \povs)}, \qquad
\eff_{\textrm{prob}}^{\ \mathrm{adv}} = \frac{e_\mathrm{proj}(\fV)}{\eta_{\textrm{prob}}^{\mathrm{adv}}(\Xs, k, \povs, \qpow)}
\end{align*} for the probabilistic error bounds 

\begin{align*}
\eta_{\textrm{prob}}(\Xs, k, \povs) &:= \left(1 + \sqrt{1 + 
6\ns/(k+\povs)} \right)\sqrt{\sum\limits_{j \geq k+1} \sigma_j^2} \\
\eta_{\textrm{prob}}^{\mathrm{adv}}(\Xs,k, \povs, \qpow) &:= 
\sqrt{\left(\sum\limits_{j \geq k+1} \sigma_j^2 \right)+ \alpha^2 6\ns/(k+\povs)}.
\end{align*} 

\begin{figure}[h]
\begin{tikzpicture}

\definecolor{darkgray176}{RGB}{176,176,176}
\definecolor{darkorange25512714}{RGB}{255,127,14}
\definecolor{forestgreen4416044}{RGB}{44,160,44}
\definecolor{lightgray204}{RGB}{204,204,204}
\definecolor{steelblue31119180}{RGB}{31,119,180}

\begin{groupplot}[group style={group size=2 by 2}, scale = 0.7]
\nextgroupplot[
legend to name=leg2,
legend cell align={left},
log basis x={10},
log basis y={10},
tick align=outside,
tick pos=left,
title={\scriptsize Effectivity prob. error bound, \(\displaystyle q_{pow}\) = 0},
x grid style={darkgray176},
xlabel={basis size},
xmajorgrids,
xmin=8.70550563296124, xmax=183.791736799526,
xmode=log,
xtick style={color=black},
xtick={0.1,1,10,100,1000,10000},
xticklabels={
  \(\displaystyle {10^{-1}}\),
  \(\displaystyle {10^{0}}\),
  \(\displaystyle {10^{1}}\),
  \(\displaystyle {10^{2}}\),
  \(\displaystyle {10^{3}}\),
  \(\displaystyle {10^{4}}\)
},
y grid style={darkgray176},
ylabel={effectivity},
ymajorgrids,
ymin=0.5, ymax=246.739615163368,
ymode=log,
ytick style={color=black},
ytick={0.01,0.1,1,10,100,1000,10000},
yticklabels={
  \(\displaystyle {10^{-2}}\),
  \(\displaystyle {10^{-1}}\),
  \(\displaystyle {10^{0}}\),
  \(\displaystyle {10^{1}}\),
  \(\displaystyle {10^{2}}\),
  \(\displaystyle {10^{3}}\),
  \(\displaystyle {10^{4}}\)
}
]
\addlegendentry{$\eff_{\textrm{prob}}^{\ \mathrm{adv}}$, $p_\text{pow} = 5$}
\addlegendentry{$\eff_{\textrm{prob}}^{\ \mathrm{adv}}$, $p_\text{pow} = 20$}
\addlegendentry{$\eff_{\textrm{prob}}^{\ \mathrm{adv}}$, $p_\text{pow} = l-k$}
\addlegendentry{$\eff_{\textrm{prob}}$, $p_\text{pow} = 5$}
\addlegendentry{$\eff_{\textrm{prob}}$, $p_\text{pow} = 20$}
\addlegendentry{$\eff_{\textrm{prob}}$, $p_\text{pow} = l-k$}
\addplot [steelblue31119180, dash pattern=on 3pt off 4.95pt]
table {%
10 38.7947526326341
20 64.5140556415834
40 67.428130314958
80 66.5598568171257
160 20.204099534075
};
\addplot [darkorange25512714, dash pattern=on 3pt off 4.95pt]
table {%
10 3.54333142655757
20 11.4606491319383
40 24.1827771622308
80 37.0549179565341
160 17.9917184221502
};
\addplot [forestgreen4416044, dash pattern=on 3pt off 4.95pt]
table {%
10 0.99999999999989
20 1.0000000000002
40 1.00000000000003
80 0.999999999999919
160 1.00775972965664
};
\addplot [steelblue31119180]
table {%
10 82.2465383877893
20 63.9364759102511
40 47.9148164229609
80 35.1424352141547
160 25.7132931848017
};
\addplot [darkorange25512714]
table {%
10 58.4543296888865
20 50.7594212185091
40 41.6324993078212
80 32.4801524773918
160 24.6632960724964
};
\addplot [forestgreen4416044]
table {%
10 9.0593727033856
20 7.32511598301566
40 6.01226565712009
80 4.96997453120763
160 4.16443722931536
};

\nextgroupplot[
log basis x={10},
log basis y={10},
tick align=outside,
tick pos=left,
title={\scriptsize Effectivity prob. error bound, \(\displaystyle q_{pow}\) = 2},
x grid style={darkgray176},
xlabel={basis size},
xmajorgrids,
xmin=8.70550563296124, xmax=183.791736799526,
xmode=log,
xtick style={color=black},
xtick={0.1,1,10,100,1000,10000},
xticklabels={
  \(\displaystyle {10^{-1}}\),
  \(\displaystyle {10^{0}}\),
  \(\displaystyle {10^{1}}\),
  \(\displaystyle {10^{2}}\),
  \(\displaystyle {10^{3}}\),
  \(\displaystyle {10^{4}}\)
},
y grid style={darkgray176},
ylabel={effectivity},
ymajorgrids,
ymin=0.5, ymax=246.739615163368,
ymode=log,
ytick style={color=black},
ytick={0.01,0.1,1,10,100,1000,10000},
yticklabels={
  \(\displaystyle {10^{-2}}\),
  \(\displaystyle {10^{-1}}\),
  \(\displaystyle {10^{0}}\),
  \(\displaystyle {10^{1}}\),
  \(\displaystyle {10^{2}}\),
  \(\displaystyle {10^{3}}\),
  \(\displaystyle {10^{4}}\)
}
]
\addplot [steelblue31119180, dash pattern=on 3pt off 4.95pt]
table {%
10 1.00340245317204
20 3.45955779513235
40 14.8418853990524
80 28.8416549926749
160 12.7980910192089
};
\addplot [darkorange25512714, dash pattern=on 3pt off 4.95pt]
table {%
10 1.00000000000135
20 1.0000010841103
40 1.01203740616617
80 2.35378557224861
160 8.54628161928923
};
\addplot [forestgreen4416044, dash pattern=on 3pt off 4.95pt]
table {%
10 0.999999999999889
20 1.00000000000021
40 1.00000000000004
80 0.999999999999931
160 0.979716935000981
};
\addplot [steelblue31119180]
table {%
10 82.2465383877893
20 63.9364759102514
40 47.9148164229613
80 35.1424352141551
160 24.9977728286273
};
\addplot [darkorange25512714]
table {%
10 58.4543296888864
20 50.7594212185093
40 41.6324993078216
80 32.4801524773921
160 23.9769938449521
};
\addplot [forestgreen4416044]
table {%
10 9.0593727033856
20 7.32511598301569
40 6.01226565712014
80 4.96997453120769
160 4.04855399381649
};

\nextgroupplot[
log basis x={10},
log basis y={10},
tick align=outside,
tick pos=left,
title={\scriptsize Effectivity prob. error bound, \(\displaystyle q_{pow}\) = 5},
x grid style={darkgray176},
xlabel={basis size},
xmajorgrids,
xmin=8.70550563296124, xmax=183.791736799526,
xmode=log,
xtick style={color=black},
xtick={0.1,1,10,100,1000,10000},
xticklabels={
  \(\displaystyle {10^{-1}}\),
  \(\displaystyle {10^{0}}\),
  \(\displaystyle {10^{1}}\),
  \(\displaystyle {10^{2}}\),
  \(\displaystyle {10^{3}}\),
  \(\displaystyle {10^{4}}\)
},
y grid style={darkgray176},
ylabel={effectivity},
ymajorgrids,
ymin=0.5, ymax=246.739615163368,
ymode=log,
ytick style={color=black},
ytick={0.01,0.1,1,10,100,1000,10000},
yticklabels={
  \(\displaystyle {10^{-2}}\),
  \(\displaystyle {10^{-1}}\),
  \(\displaystyle {10^{0}}\),
  \(\displaystyle {10^{1}}\),
  \(\displaystyle {10^{2}}\),
  \(\displaystyle {10^{3}}\),
  \(\displaystyle {10^{4}}\)
}
]
\addplot [steelblue31119180, dash pattern=on 3pt off 4.95pt]
table {%
10 1.00000000003277
20 1.00074194975963
40 1.82302131060723
80 8.27648238039454
160 6.70544054073038
};
\addplot [darkorange25512714, dash pattern=on 3pt off 4.95pt]
table {%
10 0.99999999999989
20 1.0000000000002
40 1.00000000323588
80 1.00043204220991
160 3.01293419430733
};
\addplot [forestgreen4416044, dash pattern=on 3pt off 4.95pt]
table {%
10 0.99999999999989
20 1.0000000000002
40 1.00000000000003
80 0.999999999999953
160 0.970995251432207
};
\addplot [steelblue31119180]
table {%
10 82.2465383877894
20 63.9364759102513
40 47.9148164229609
80 35.1424352141558
160 24.7752364441407
};
\addplot [darkorange25512714]
table {%
10 58.4543296888865
20 50.7594212185092
40 41.6324993078212
80 32.4801524773929
160 23.7635446885936
};
\addplot [forestgreen4416044]
table {%
10 9.05937270338561
20 7.32511598301568
40 6.01226565712009
80 4.9699745312078
160 4.01251276028904
};

\nextgroupplot[
hide x axis,
hide y axis,
legend cell align={left},
legend style={fill opacity=0.8, draw opacity=1, text opacity=1, draw=lightgray204},
tick align=outside,
tick pos=left,
x grid style={darkgray176},
xmin=0, xmax=1,
xtick style={color=black},
y grid style={darkgray176},
ymin=0, ymax=1,
ytick style={color=black}
]
\end{groupplot}
\node[
   right = 9.4cm,
   below = 6.0cm,
    align=left,
    draw=black,
    anchor=south,
  ] {%
  \ref*{leg2}
  };
\end{tikzpicture}
\caption{Effectivity of probabilistic error bound for different values of $\povs$ and $\qpow$}
\label{probbound}
\end{figure}
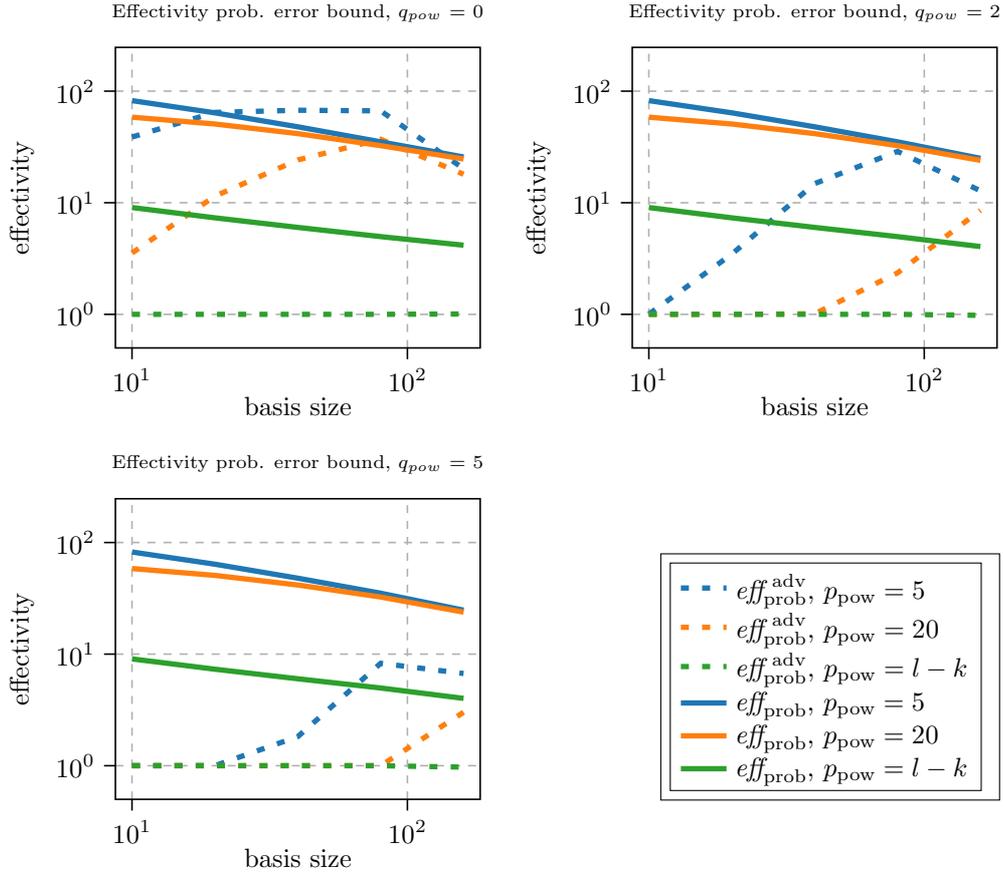
By \Cref{thmbound,thmadvbound}, the failure probability of the two bounds is $2/k$ for $\povs=l-k$, which is for the values of $k$ considered here between $20\%$ for $k = 10$ and $1.25\%$ for $k = 160.$ However, we observe in \Cref{probbound} where we present average results over 5 runs (i.e., 5 draws of $\fOmega$) that in practice the effectivities are not lower than one. Note, that we compute the effectivities also for $\povs = 5, 20$ where the assumption $\povs \geq l-k$ does not hold. Nevertheless, we observe that also in this case the effectivities are greater or equal than one. Moreover, we realize that the assumption $\povs \geq k-l$ is needed in \Cref{propSRFT} as we observe that the effectivities of the probabilistic bounds are sometimes lower than the effectivities of the deterministic bounds for $\povs = 5, 20.$ We further observe that $\eff_{\textrm{prob}}^{\ \mathrm{adv}}$ gets closer to one the higher $\qpow$ and $\povs$ are chosen and similarly $\eff_{\textrm{prob}}$ becomes close to one the higher $\povs$ is chosen.

\section{Conclusion and Outlook} \label{Conc}
In this work, we presented two probabilistic error bounds for the rcSVD basis generation procedure that depend on the choice of two hyperparameters. With a certain probability which depends on the basis size a suitable choice leads to the projection error of the rcSVD being at most a constant factor worse than the projection error of the cSVD, i.e., the rcSVD being quasi-optimal in the set of ortho-symplectic matrices. However, the numerical experiments showed that the resulting choice for the oversampling parameter $\povs$ required for having these guarantees is only useful if $k+\povs\ll \ns$. In practice, smaller values for $\povs$ also work very well where we do not have probabilistic bounds. Moreover, we learn from \Cref{thmadvbound} that the performance of the rcSVD algorithm depends on the quotient $(\sigma_k/\sigma_{k+\povs+1})^{\qpow}$. One option for future work is applying (randomized) error estimates  for the projection error and combining them with adaptive randomized basis generation. Future work will also deal with error analysis of the rSVD-like-algorithm \cite{Herkert2023rand}, a randomized version of the SVD-like decomposition \cite{Xu2003,Buchfink2019}. Another option for future work is the analysis of different complex sketching matrices, i.e., bounding the norms of $\fOmega_1, \fOmega_2$ for other random distributions. Furthermore, our implementation could be adapted and tested on different hardware (e.g. multicore architectures), as random sketching techniques are easily parallelizable and therefore well suited to modern computing architectures.

\section*{Declaration of competing interest}
The authors declare no competing interests.
\section*{Data availability}
The code for the experiments is openly available at \url{doi.org/10.18419/darus-4185}.
\section*{Acknowledgements}
Supported by Deutsche Forschungsgemeinschaft (DFG,
German Research Foundation) Project No. 314733389, and under Germany’s
Excellence Strategy - EXC 2075 – 390740016. We acknowledge the support by
the Stuttgart Center for Simulation Science (SimTech).
\bibliographystyle{elsarticle-num.bst}

\end{document}